\documentclass[article,11pt]{amsart}
\usepackage[utf8]{inputenc}

\usepackage{graphicx, amsmath, amssymb, amscd, amsthm, euscript, psfrag, amsfonts,bm}

\usepackage{enumitem}
\usepackage{etoolbox}
\usepackage{amsmath}
\usepackage{enumerate}
\usepackage{amssymb}
\usepackage{amscd}
\usepackage{amsthm}
\usepackage{amsfonts}
\usepackage{graphicx}
\usepackage{tensor}
\usepackage{multirow}
\usepackage{multicol}
\usepackage{mathscinet}
\usepackage[T1]{fontenc}
\usepackage{mathrsfs}

\usepackage[colorlinks=true]{hyperref}
\hypersetup{urlcolor=blue, citecolor=red}
\usepackage{cleveref}
\crefformat{section}{\S#2#1#3} % see manual of cleveref, section 8.2.1
\crefformat{subsection}{\S#2#1#3}
\crefformat{subsubsection}{\S#2#1#3}

\newtheorem{theorem}{Theorem}[section]
\newtheorem{corollary}[theorem]{Corollary}

\newtheorem{lemma}[theorem]{Lemma}

\newtheorem{proposition}[theorem]{Proposition}

\theoremstyle{definition}
\newtheorem{definition}[theorem]{Definition}
\newtheorem{example}[theorem]{Example}
\newtheorem{remark}[theorem]{Remark}

\newcommand{\norm}[1]{\left\lVert#1\right\rVert}
\DeclareMathOperator\diag{diag}
\DeclareMathOperator\dist{dist}
\DeclareMathOperator\supp{supp}
\DeclareMathOperator\Lip{Lip}

\title[Pointwise equidistribution and Weighed L\'evy-Khintchin theorem]{Pointwise equidistribution for almost smooth functions with an error rate and Weighted L\'evy-Khintchin theorem}

\author{Bohan Yang}
\address{Yau Mathematical Sciences Center, Tsinghua University, Beijing, 100084, China }
\email{ybh20@mails.tsinghua.edu.cn}
\author{Han Zhang}
\address{School of Mathematical Science, Soochow University, Suzhou 215006, China
}
\email{hzhang.math@suda.edu.cn}
\date{}

\begin{document}

\begin{abstract}
The purpose of this article is twofold: to prove a pointwise equidistribution theorem with an error rate for almost smooth functions, which strengthens the main result of Kleinbock, Shi and Weiss (2017); and to obtain a L\'evy-Khintchin theorem for weighted best approximations, which extends the main theorem of Cheung and Chevallier (2019).

To do so, we employ techniques from homogeneous dynamics and the methods developed in the work of Cheung-Chevallier (2019) and Shapira-Weiss (2022).

\end{abstract}

\maketitle

\section{Introduction}
Let $N$ be a positive integer, $G=SL_N(\mathbb{R})$ and $\Gamma=SL_{N}(\mathbb{Z})$. We denote $\mathfrak{X}_{N}=G/\Gamma$, which is the space of unimodular lattices in $\mathbb{R}^N$. Take $m,n\in \mathbb{N}$ such that $m+n=N$. Denote by $M_{m,n}$ the space of $m\times n$ matrices. For $A\in M_{m,n}$, we consider the unimodular lattice $u(A)\mathbb{Z}^N\in \mathfrak{X}_{N}$, where 
\[u(A):=\begin{pmatrix}
    I_m & A\\
    0 & I_n
\end{pmatrix}.\]
A vector $\boldsymbol{a}={^t(}a_1,\cdots,a_m)\in \mathbb{R}_{>0}^m$ ($\boldsymbol{b}\in \mathbb{R}^n_{>0}$, resp.) is called a weight vector if $\sum_{i=1}^m a_i=1$ ($\sum_{j=1}^n b_j=1$, resp.), where the superscript $t$ means the transpose of the vector and we view $\mathbb{R}^m$ as the set of $m$ by $1$ column vectors. The case $\boldsymbol{a}={^t(}1/m,\cdots,1/m)$ and $\boldsymbol{b}={^t(}1/n,\cdots,1/n)$ will be referred as equal weights. For two weight vectors
 $\boldsymbol{a}\in \mathbb{R}_{>0}^m$ and $\boldsymbol{b}\in \mathbb{R}_{>0}^n$, we define a one-parameter diagonal subgroup $\{a_t\}_{t\in \mathbb{R}}\subset G$ by 
 \begin{equation}\label{equation: definition of a t}
     a_t:=\diag(e^{a_1 t},\cdots,e^{a_m t},e^{-b_1 t},\cdots,e^{-b_n t}).
 \end{equation}
It turns out that many weighted Diophantine properties of the matrix $A$ are encoded by the $a_t$-trajectory of $u(A)\mathbb{Z}^N$ in $\mathfrak{X}_N$. For a small sample of such deep connections, we refer the reader to the classical Dani's correspondence \cite{Dani_1985_Divergent_trajectories_of_flows_on_homogeneous_spaces_MR794799} and its developments (see e.g. \cite{Kleinbock_Weiss_2013_Modified_Schmidt_games_and_a_conjecture_of_Margulis_MR3296561,Kleinbock_Weiss_2008_Dirichlet's+theorem+on+Diophantine_MR2366229,Saxce_2022_Rational+approximations+to+linear+subspaces}), the recent theory of parametric geometry of numbers developed by Schmidt-Sumerer \cite{Schmidt_Summerer_2009_Parametric_geometry_of_numbers_and_applications_MR2557854,Schmidt_Summerer_2013_Diophantine_approximation_and_parametric_geometry_of_numbers_MR3016519} and Roy \cite{Roy_2015_On_Schmidt_and_Summerer_parametric_goemMR3418530,Roy_2016_Spectrum_of_the_exponents_of_best_rational_approximation_MR3489062}. 

Usually for a given matrix $A$, it is difficult to analyze the $a_t$-trajectory of $u(A)\mathbb{Z}^{N}$. However, using tools from ergodic theory and homogeneous dynamics, it is possible to anticipate the behavior of $a_t$-trajectory of many $u(A)\mathbb{Z}^N$, in the sense of both Lebesgue measure and Hausdorff dimension (see e.g. \cite{Cheung_2011_Hausdorff_dimension_of_the_set_of_singular_pairs_MR2753601,Cheung_Chevallier_2016_Hausdorff_dimension_of_singular_vectors_MR3544282, Liao_Shi_Solan_Tamam_2020_Hausdorff_dimension_of_weighted_MR4055990,Kleinbock_Margulis_1999_Logarithm_laws_MR1719827,Kleinbock_Weiss_2013_Modified_Schmidt_games_and_a_conjecture_of_Margulis_MR3296561,Shi_2020_Pointwise_equidistribution_for_one_parameter_diagonalizable_group_action_on_homogeneous_space_MR4105521,Kleinbock_Strombergsson_Yu_2022_A_measure_estimate_MR4500198}). See also \cite{Shi_Weiss_2017_Invariant+measures+for+solvable+groups+MR3642031,Shah_2009_Equidistribution_of_expanding_translates_of_curves_and_MR2534098,Yang_2020_Equidistribution_of_expanding_translates_of_curves_and_MR4094972} for equidistribution results with respect to measures supported on submanifolds of $M_{m,n}$. In particular, when $\boldsymbol{a},\boldsymbol{b}$ are of equal weights, the classical Birkhoff ergodic theory implies that for Lebesgue almost every $A$, the $a_t$-trajectory of $u(A)\mathbb{Z}^N$ equidistributes in $\mathfrak{X}_N$ with respect to the unique $G$-invariant probability measure. For general weight vectors, Birkhoff ergodic theorem does not apply. Yet using techniques from homogeneous dynamics, Shi \cite{Shi_2020_Pointwise_equidistribution_for_one_parameter_diagonalizable_group_action_on_homogeneous_space_MR4105521} and Kleinbock-Shi-Weiss \cite{Kleinbock_Shi_Barak_2017_Pointwise_equidistribution_with_an_error_rate_and_with_respect_to_unbounded_functions_MR3606456} obtained the same equidistribution result for Lebesgue almost every $A$. More precisely, Shi \cite{Shi_2020_Pointwise_equidistribution_for_one_parameter_diagonalizable_group_action_on_homogeneous_space_MR4105521} obtained the pointwise equidistribution result in the setting of general Lie group modulo its lattice. While \cite{Kleinbock_Shi_Barak_2017_Pointwise_equidistribution_with_an_error_rate_and_with_respect_to_unbounded_functions_MR3606456} focuses only on $SL_N(\mathbb{R})/SL_N(\mathbb{Z})$, an equidistribution result with error rate for compactly supported smooth functions was obtained. Our first main theorem of this article strengthens the main result of \cite{Kleinbock_Shi_Barak_2017_Pointwise_equidistribution_with_an_error_rate_and_with_respect_to_unbounded_functions_MR3606456} to the case of so called `almost smooth' functions on $SL_N(\mathbb{R})/SL_N(\mathbb{Z})$. The interested reader may also see \cite{Gaposhkin_1981_The+dependence+of+the+rate_MR636767,Kach_1996_Rates+of+convergence_MR1422228} for relevant results on ergodic average with an error rate.

\begin{theorem}\label{Theorem: Pointwise ergodic theorem with rate for nonsmooth functions}
Let $\Lambda\in \mathfrak{X}_N$ and $\epsilon>0$ be given. Let $f$ be an almost smooth function on $\mathfrak{X}_N$. Then for Lebesgue almost every $A\in \mathcal{U}$, 
\begin{align*}
    \frac{1}{T}\int_0^T f(a_t u(A)\Lambda)dt=\mu_N(f)+o(T^{-1/2}\log^{2+\epsilon}T).
\end{align*}
Here the notation $\phi(T)=o(\psi(T))$ means $\lim_{T\to \infty}\frac{\phi(T)}{\psi(T)}=0$.
\end{theorem}
The notion of almost smooth functions is taken from \cite[Section 2]{Kleinbock_Margulis_1996_Bounded_orbits_of_nonquasiunipotent_MR1359098}. A function $f$ is `almost smooth' if it can be approximated reasonably well by suitable smooth functions from above and below, see Definition \ref{Definition: almost smooth functions} for precise meaning of `reasonably well'. Many nice functions are almost smooth, see Examples \ref{Example: almost smooth function 1} and \ref{Example: almost smooth funcions 2}. We remark that any $f\in C_c^{\infty}(\mathfrak{X}_N)$ is automatically almost smooth. 

\begin{remark}
\begin{itemize}
    \item[(1)]
\cite[Theorem 1.1]{Kleinbock_Shi_Barak_2017_Pointwise_equidistribution_with_an_error_rate_and_with_respect_to_unbounded_functions_MR3606456} holds for compactly supported smooth functions. Since the proof of \cite[Theorem 1.1]{Kleinbock_Shi_Barak_2017_Pointwise_equidistribution_with_an_error_rate_and_with_respect_to_unbounded_functions_MR3606456} depends on a Borel-Cantelli type argument, approximating an almost smooth function $f$ using smooth functions and applying \cite[Theorem 1.1]{Kleinbock_Shi_Barak_2017_Pointwise_equidistribution_with_an_error_rate_and_with_respect_to_unbounded_functions_MR3606456} directly to these smooth functions would not yield Theorem \ref{Theorem: Pointwise ergodic theorem with rate for nonsmooth functions}. The proof of Theorem \ref{Theorem: Pointwise ergodic theorem with rate for nonsmooth functions} requires a generalization of \cite[Theorem 3.1]{Kleinbock_Shi_Barak_2017_Pointwise_equidistribution_with_an_error_rate_and_with_respect_to_unbounded_functions_MR3606456}, which is Theorem \ref{Theorem: a general method for proving pointwise ergodic with rate} in this article.

\item[(2)] Theorem \ref{Theorem: Pointwise ergodic theorem with rate for nonsmooth functions} is a strengthening of \cite[Theorem 1.1]{Kleinbock_Shi_Barak_2017_Pointwise_equidistribution_with_an_error_rate_and_with_respect_to_unbounded_functions_MR3606456}, although with a slightly worse error rate ($T^{-1/2}\log^{2+\epsilon}T$ instead of $T^{-1/2}\log^{3/2+\epsilon}T$).
\end{itemize}
\end{remark}
As a second main theme of this article, we present an number 
theoretic application of Theorem \ref{Theorem: Pointwise ergodic theorem with rate for nonsmooth functions} (or \cite[Theorem 1.1]{Kleinbock_Shi_Barak_2017_Pointwise_equidistribution_with_an_error_rate_and_with_respect_to_unbounded_functions_MR3606456}) concerning the recent work of Cheung-Chevallier \cite{Cheung_Chevallier_2019_Levy_Khintchin_Theorem} and Shapira-Weiss \cite{Shapira_Weiss_2022_Geometric_and_arithmetic_aspects}. We refer the reader to \cite{Kleinbock_Shi_Barak_2017_Pointwise_equidistribution_with_an_error_rate_and_with_respect_to_unbounded_functions_MR3606456} for other interesting applications.

To illustrate the application of pointwise equidistribution result to number theory, we let $d$ be a positive integer and  $\boldsymbol{w}\in \mathbb{R}^d_{>0}$ be a weight vector. We define a $\boldsymbol{w}$-quasi-norm (see e.g. \cite{Liao_Shi_Solan_Tamam_2020_Hausdorff_dimension_of_weighted_MR4055990}) on $\mathbb{R}^d$ by
\begin{equation}\label{equation: definition of w-quasi norm}
    \|{^t(}x_1,x_2,\cdots,x_d)\|_{\boldsymbol{w}}=\max_{1\leq i\leq d}\{|x_i|^{1/{w_i}}\}.
\end{equation}
For a vector $\boldsymbol{\theta}\in \mathbb{R}^d$, one can associate a sequence of $\boldsymbol{w}$-best approximable vectors $(\boldsymbol{p}_n(\boldsymbol{\theta}),q_n(\boldsymbol{\theta}))\in \mathbb{Z}^d\times \mathbb{N}$ to $\boldsymbol{\theta}$ with respect to $\boldsymbol{w}$-quasi-norm (for precise definition, see Definition \ref{Definition: w best approximation vector}). Roughly speaking, $(\boldsymbol{p}_n(\boldsymbol{\theta}),q_n(\boldsymbol{\theta}))\in \mathbb{Z}^d\times \mathbb{N}$ is a $\boldsymbol{w}$-best approximable vector of $\boldsymbol{\theta}$ if among all rational vectors of denominator no greater than $q(\boldsymbol{\theta})$, $\boldsymbol{p}(\boldsymbol{\theta})/q(\boldsymbol{\theta})$ is the closest rational vector to $\boldsymbol{\theta}$ with respect to the norm $\norm{\cdot}_{\boldsymbol{w}}$. When $d=1$, $\{p_n(\theta)/q_n(\theta)\}$ are exactly convergents of $\theta\in \mathbb{R}$ in the classical theory of continued fraction expansion. Thus, the theory of best approximation is a higher dimensional generalization of continued fraction theory. 

For the classical continued fraction theory, the growth rate of $q_n$ has attracted considerable attention. Khintchin \cite{Khintchine_1936_Zur+metrischen_MR1556944} showed that for almost all real numbers $\theta$, denominators $\{q_n\}$ of its convergent satisfy 
\[\lim_{n\to \infty} q_n^{1/n}=\gamma\]
for some constant $\gamma$, whose exact value was given by L\'evy \cite{Levy_1936_Sur+le+development_MR1556945}. Using the connection between Gauss map and continued fraction expansion, L\'evy-Khintchin theorem could be proved using Birkhoff ergodic theorem\footnote{ Khintchin didn't use ergodic theory to prove his result.}. However, there is no corresponding `Gauss map' for best approximations in higher dimensional setting. We refer the interested reader to \cite{Cheung_Chevallier_2019_Levy_Khintchin_Theorem} and the references therein for a more detailed historical account for the study of L\'evy-Khintchin type result from different aspects.

Recently, Cheung and Chevallier \cite{Cheung_Chevallier_2019_Levy_Khintchin_Theorem} obtained a L\'evy-Khintchin theorem for best approximations in higher dimensional setting, and calculated the exact value of the corresponding L\'evy-Khintchin constant when $d=2$. They circumvented this problem by analyzing the first return map of a diagonal flow to certain cross-section in the space of unimodular lattices. Soon afterward, Shapira and Weiss \cite{Shapira_Weiss_2022_Geometric_and_arithmetic_aspects} promoted Cheung-Chevallier's result to Adelic setting, and the asymptotic statistical properties of different types of vectors $\boldsymbol{\theta}$ were obtained. Certain cross-section in the space of unimodular lattices also plays a crucial role in the work of Shapira-Weiss, though this cross-section is different from the one used in Cheung-Chevallier.

Following the approach of Shapira and Weiss in \cite{Shapira_Weiss_2022_Geometric_and_arithmetic_aspects}, we extends L\'evy-Khintchin theorem to $\boldsymbol{w}$-best approximations.
\begin{theorem}\label{Theorem: weighted Levy-Khintchine constant}
Let $\boldsymbol{w}\in \mathbb{R}^d$ be a weight vector. Then there is a constant $L_d(\boldsymbol{w})$ such that for almost all $\boldsymbol{\theta}\in \mathbb{R}^d$, 
\begin{itemize}
    \item[(1)] \begin{equation*}\label{equation: limit of qn}
    \lim_{n\to \infty}\frac{1}{n}\log q_n(\boldsymbol{\theta})=L_d(\boldsymbol{w}),
\end{equation*}

    \item[(2)] \begin{equation*}\label{equation: limit of rn}
    \lim_{n\to \infty}\frac{1}{n}\log \norm{q_n(\boldsymbol{\theta})\boldsymbol{\theta}-\boldsymbol{p}_n(\boldsymbol{\theta})}_{\boldsymbol{w}}=-L_d(\boldsymbol{w}).
\end{equation*}
\item[(3)]  Moreover, denote $\beta_n(\boldsymbol{\theta})=q_{n+1}(\boldsymbol{\theta})\norm{q_n(\boldsymbol{\theta})\boldsymbol{\theta}-\boldsymbol{p}_n(\boldsymbol{\theta})}_{\boldsymbol{w}}$ for $\boldsymbol{\theta}\in \mathbb{R}^d$, there exists a probability measure $\nu_d^{\boldsymbol{w}}$ on $\mathbb{R}$ such that for almost all $\boldsymbol{\theta}$, 
\begin{equation*}\label{equation: distribution of qn times discrepancy}
    \lim_{n\to \infty}\frac{1}{n}\sum_{k=1}^n \delta_{\beta_k(\boldsymbol{\theta})}=\nu_d^{\boldsymbol{w}} \text{ in the weak-* topology,}
\end{equation*}
where $\delta_{\beta_k(\boldsymbol{\theta})}$ denotes the Dirac measure on $\beta_k(\boldsymbol{\theta})$.
\end{itemize}
\end{theorem}
Note that when $\boldsymbol{w}={^t(}1/d,\cdots,1/d )$ is of equal weight, Theorem \ref{Theorem: weighted Levy-Khintchine constant} recovers \cite[Theorem 1]{Cheung_Chevallier_2019_Levy_Khintchin_Theorem} in the case of best approximations for vectors. In this article we are not intended to compute the constant $L_d(\boldsymbol{w})$. We remark that to obtain Theorem \ref{Theorem: weighted Levy-Khintchine constant}, following the strategy of \cite{Shapira_Weiss_2022_Geometric_and_arithmetic_aspects} we consider a suitable open subset $\mathcal{B}$ (see (\ref{align: definition of open set B in the cross section})) of certain cross-section in $\mathfrak{X}_{d+1}$ and investigate the visit times of a (weighted) diagonal flow to $\mathcal{B}$. The proof combines the ideas from \cite{Shapira_Weiss_2022_Geometric_and_arithmetic_aspects}, \cite[Section 6]{Cheung_Chevallier_2019_Levy_Khintchin_Theorem} and Kleinbock-Margulis approach \cite{Kleinbock_Margulis_1999_Logarithm_laws_MR1719827} (see also \cite{Kleinbock_Strombergsson_Yu_2022_A_measure_estimate_MR4500198}).

Theorem \ref{Theorem: weighted Levy-Khintchine constant} is relevant to Question 3 raised in \cite[Section 10]{Cheung_Chevallier_2019_Levy_Khintchin_Theorem}, yet we caution the reader that Theorem \ref{Theorem: weighted Levy-Khintchine constant} did not give a solution to this question. In the appendix we will reveal the difference between Theorem \ref{Theorem: weighted Levy-Khintchine constant} and Question 3 in \cite{Cheung_Chevallier_2019_Levy_Khintchin_Theorem}.

\subsection{Notations and conventions}

\begin{itemize}
    \item Throughout this article, for two quantities $A$ and $B$, by $A\ll B$ or $A=O(B)$ we mean that there is a constant $c>0$ such that $A\leq cB$. We will use subscripts to indicate the dependence of the constant on parameters. The notation $A\asymp B$ means $A\ll B\ll A$. 

\item Unless specified, we will use boldsymbol letters such as $\boldsymbol{v}$ to denote vectors in Euclidean space. All vectors in this article are assumed to be column vectors. For a column vector $\boldsymbol{v}$, $^t\boldsymbol{v}$ will be its transpose.

\item For a topological space $X$, and any subset $E\subset F\subset X$, $E^0$ is the interior of $E$ with respect to the topology of $X$. $cl_F(E)$ is the closure of $E$ with respect to the induced topology in $F$. $\partial_F E$ is the topological boundary of $E$ in $F$. Sometimes we will write $cl_F(E)$ as $cl(E)$, and $\partial_F E$ as $\partial E$ when $F$ is clear in the context.

\end{itemize}

\subsection{Overview of the paper}
In Section \ref{Section: Preliminaries}, we collect some known results which will be used in later sections. In Section \ref{Section: pointwise ergodic theorem}, we prove the pointwise equidistribution theorem for almost smooth functions with an error rate. In Section \ref{Section: cross sections}, we recall the techniques developed by Shapira-Weiss \cite{Shapira_Weiss_2022_Geometric_and_arithmetic_aspects}, and then define a suitable cross-section and investigate visit times of a weighted diagonal flow to certain subset of this cross-section. The main result of this section is that Theorem \ref{Theorem: weighted Levy-Khintchine constant} (1) holds. Section \ref{Section: Borel-Cantelli} is contributed to proving (2) (3) of Theorem \ref{Theorem: weighted Levy-Khintchine constant}. In the appendix, we will illustrate the difference between Theorem \ref{Theorem: weighted Levy-Khintchine constant} and Question 3 in \cite{Cheung_Chevallier_2019_Levy_Khintchin_Theorem}.

\noindent$\textbf{Acknowledgment.}$
We thank Professor Yitwah Cheung for helpful discussions.

\section{Preliminaries}\label{Section: Preliminaries}

In this section, we collect some known results which we will need in later sections. Let $m,n$ be two positive integers and $N=m+n$. Let $G=SL_{N}(\mathbb{R})$, $\Gamma=SL_N(\mathbb{Z})$, and $\mathfrak{X}_N=G/\Gamma$. We denote by $\mu_N$ the unique $G$-invariant probability measure on $\mathfrak{X}_N$.

\subsection{Sobolev norm on $\mathfrak{X}_N$}

Let $\mathfrak{g}=\mathfrak{sl}_N$ be the Lie algebra of $G$. For each $Y\in \mathfrak{g}$, let $\mathcal{D}_Y$ be the corresponding Lie derivative on $C^{\infty}(G)$ defined by 
\[\mathcal{D}_Y(f)(g):=\frac{d}{dt}f(\exp(tY)g)|_{t=0}, \quad \forall f\in C^{\infty}(G),\]
where $\exp:\mathfrak{g}\to G$ denotes the usual exponential map from $\mathfrak{g}$ to $G$. This operator extends to $C_c^{\infty}(\mathfrak{X}_N)$ naturally. For $a\in \mathbb{N}$ and $Y_1,\cdots,Y_a\in \mathfrak{g}$, the monomial $Z=Y_1^{l_1}\cdots Y_a^{l_a}$ defines a differential operator of degree $\deg(Z)=l_1+\cdots+l_a$ by
\[\mathcal{D}_Z:=D_{Y_1}^{l_1}\circ \cdots \circ \mathcal{D}_{Y_a}^{l_a},\]
where $\mathcal{D}_{Y_i}^{l_i}$ means $\mathcal{D}_{Y_i}$ is composed with itself $l_i$-times. 

Fix an basis $\{Y_1,\cdots,Y_N\}$ of $\mathfrak{g}$. For any $l\in \mathbb{N}$, we define the Sobolev norm on $C_c^{\infty}(\mathfrak{X}_N)$ by 
\begin{equation}\label{equation: definition of Sobolev norm}
    \norm{f}_{H^l}:=\left( \sum_{\deg(Z)\leq l}\int_{\mathfrak{X}_N}|\mathcal{D}_Z(f)|^2 d\mu_N \right)^{1/2},\quad \forall f\in C_c^{\infty}(\mathfrak{X}_N),
\end{equation}
where the summation is over all monomials $Z$ in $\{Y_1,\cdots,Y_N\}$ with degree no greater than $l$.

Let $\dist_G(\cdot,\cdot)$ be a right invariant Riemannian metric on $G$. It induces a Riemannian metric on $\mathfrak{X}_N$, which we denote it by $\dist(\cdot,\cdot)$. We define the Lipschitz norm on $C_c^{\infty}(\mathfrak{X}_N)$ with respect to $\dist(\cdot,\cdot)$ by
\[\norm{f}_{\Lip}:=\sup \left\{\frac{|f(x_1)-f(x_2)|}{\dist(x_1,x_2)}: x_1,x_2\in \mathfrak{X}_N, x_1\neq x_2 \right\}, \quad f\in C_c^{\infty}(\mathfrak{X}_N).\]
We also denote $\norm{\cdot}_{\infty}$ to be the uniform norm (supremum norm) on $C_c(\mathfrak{X}_N)$. For any $f\in C_c^{\infty}(\mathfrak{X}_N)$, and $l\in \mathbb{N}$, the following quantity will be important later:
\begin{equation}\label{equation: definition of Nl}
    \mathcal{N}_l(f):=\max\left\{\norm{f}_{\infty}, \norm{f}_{\Lip},\norm{f}_{H^l}\right\}.
\end{equation}

\subsection{Effective equidistribution and doubly mixing}\label{section: equidistribution and doubly mixing}
Let $\boldsymbol{a}\in \mathbb{R}^m$ and $\boldsymbol{b}\in \mathbb{R}^n$ be weight vectors. Define the one-parameter diagonal subgroup $\{a_t:t\in \mathbb{R}\}$ of $G$ as in (\ref{equation: definition of a t}).
Denote
\begin{align}
    \mathcal{U}:=\left\{u(A):=\begin{pmatrix}
    I_m & A\\
    0 & I_n
\end{pmatrix} \in G: A\in M_{m,n}\right\} ,\label{align: definition of mathcal U}
\end{align}
and
\begin{align*}
    \mathcal{Y}:=\{u(A) \text{ mod }\Gamma: A\in M_{m,n}(\mathbb{R})\}\cong M_{m,n}(\mathbb{R}/\mathbb{Z}).
\end{align*}
We record the following effective equidistribution and doubly mixing result for $a_t$-translation of $\mathcal{U}$ (and $\mathcal{Y}$).
\begin{theorem}(\cite[Theorem 2.2]{Bjorklund_Gorodnik_2019_Central_limit_theorems_MR3985114})\label{Theorem: effective mixing}
There exists $l\in \mathbb{N}$ and $\delta>0$ such that for any compact $\Omega\subset \mathcal{U}$, $\phi\in C_c^{\infty}(\mathcal{U})$ with $\supp \phi\subset \Omega$, $f, f_1,f_2\in C_c^{\infty}(\mathfrak{X}_N)$, $\Lambda\in \mathfrak{X}_N$ and $t,t_1,t_2>0$,
\begin{align*}
     \int_{\mathcal{U}} f(a_t u(A) \Lambda)\phi(A)dA=\mu_N(f)\cdot\int_{\mathcal{U}}\phi dA +O_{\Lambda,\Omega,\phi}(e^{-\delta t} \mathcal{N}_l(f));\\
\int_{\mathcal{U}} f_1(a_{t_1}u(A)\Lambda)f_2(a_{t_2}u(A)\Lambda)\phi(A)dA=\mu_{N}(f_1) \mu_N(f_2)\cdot\int_{\mathcal{U}}\phi dA\\ +O_{\Lambda,\Omega,\phi}(e^{-\delta\min\{t_1,t_2,|t_1-t_2|\}} \mathcal{N}_l(f_1)\mathcal{N}_l(f_2)).
\end{align*}
\end{theorem}

\begin{corollary}\label{Corollary: effective equidistribution}\cite[Corollary 6.4]{Kleinbock_Strombergsson_Yu_2022_A_measure_estimate_MR4500198}
Let $l\in \mathbb{N}$ and $\delta>0$ be as in Theorem \ref{Theorem: effective mixing}. Then for any $f\in C_c^{\infty}(\mathfrak{X}_N)$ and $t>0$,
\begin{equation*}
    \int_{\mathcal{Y}}f(a_t u(A)\mathbb{Z}^{d+1})dA=\mu_N(f)+O(e^{-\delta t}\mathcal{N}_l(f)).
\end{equation*}
\end{corollary}
Now for any $r>0$, consider a neighborhood of identity element in $G$ defined by 
\begin{equation}\label{equation: definition of neighborhood of G}
    \mathcal{O}_r:=\{g\in G: \max\{\norm{g-Id},\norm{g^{-1}-Id}\}< r\},
\end{equation}
where $\norm{\cdot}$ is the supremum norm on $M_N(\mathbb{R})$. We remark that there exists a sufficiently small neighborhood $\mathcal{V}$ of $Id$ in $G$ such that for any $g\in \mathcal{V}$,
\[\norm{g-Id}\asymp_d \norm{g^{-1}-Id}\asymp_d \dist_G(g,Id),\]
where $\dist_G$ is a fixed right invariant Riemannian metric on $G$.

\begin{lemma}\label{Lemma: inclusion about Or}
Given $0<r,\epsilon<1$, we have 
\begin{equation}\label{equation: inclusions of neighborhood}
    \mathcal{O}_{\epsilon}\cdot \mathcal{O}_r\subset \mathcal{O}_{r+\epsilon+r\epsilon}.
\end{equation}

\end{lemma}

\begin{proof}
Let $g\in \mathcal{O}_r, h\in \mathcal{O}_{\epsilon}$. Note that 
\[1-\epsilon\leq \norm{h}\leq 1+\epsilon.\]
We have
\begin{align*}
    \norm{hg-Id}&=\norm{hg-h+h-Id}\leq \norm{h}\norm{g-Id}+\norm{h-Id}\leq r+\epsilon+r\epsilon.
\end{align*}
The inequality $\norm{g^{-1}h^{-1}-Id}\leq r+\epsilon+r\epsilon$ is proved similarly. So (\ref{equation: inclusions of neighborhood}) follows.

\end{proof}

\begin{lemma}\label{Lemma: bump function}\cite[Lemma 6.6]{Kleinbock_Strombergsson_Yu_2022_A_measure_estimate_MR4500198}
For every $0<r<1$, there exists a function $\theta_r\in C_c^{\infty}(G)$ such that $\theta_r\geq 0$, $\supp(\theta_r)\subset \mathcal{O}_r$, $\int_G \theta_r(g)d\mu_G(g)=1$, and $\norm{\mathcal{D}_Z(\theta_r)}_{L^{\infty}(G)}\ll_{l,N} r^{1-N^2-l}$ for every monomial $Z=Y_1^{l_1}\cdots Y_a^{l_a}$, where $\mu_G$ is the Haar measure on $G$, $l=\deg (Z)=l_1+\cdots+l_a$.
    
\end{lemma}

\subsection{Measurable cross-sections and tempered sets}

The followings are some important notions developed in \cite{Shapira_Weiss_2022_Geometric_and_arithmetic_aspects}.

Throughout we will let $X$ be a locally compact second countable topological space (lcsc) with Borel $\sigma$-algebra $\mathcal{B}_X$. We will let $\{a_t:t\in \mathbb{R}\}$ be a one-parameter group acting measurably on $X$, that is, the map $(t,x)\mapsto a_t x$ from $\mathbb{R}\times X\to X$ is measurable.

\begin{definition}\label{Definition: Borel cross section}
  A Borel cross-section is a Borel subset $\mathcal{S}$ of $X$ with the following properties:
\begin{itemize}
    \item[(i)]For any $x\in X$, the set of visit times
    \begin{equation*}
        \mathcal{Y}_x=\{t\in\mathbb{R}:a_t x\in\mathcal{S}\}
    \end{equation*}
    is discrete and totally unbounded; that is, for any $T>0$,
    \begin{equation*}
        \mathcal{Y}_x\cap (T,\infty)\neq\emptyset,\   \mathcal{Y}_x\cap (-\infty,-T)\neq\emptyset \
        \rm{and}\ \#  ( \mathcal{Y}_x\cap (-T,T))<\infty.
    \end{equation*}
    \item[(ii)]The return time function
    \begin{equation*}
        \tau_\mathcal{S}: \mathcal{S}\to\mathbb{R}_+,\quad  \tau_\mathcal{S}=\min (\mathcal{Y}_x\cap\mathbb{R}_+)
    \end{equation*}
    is Borel.
\end{itemize} 
\end{definition}
Let $\mathcal{S}$ be a Borel cross-section, we denote by $T_{\mathcal{S}}:\mathcal{S}\to \mathcal{S}$ the first return map, defined by
\[T_{\mathcal{S}}(x)=a_{\tau_{\mathcal{S}}(x)}x.\]
In many situations we want to consider a Borel cross-section with respect to some measure $\mu$.

\begin{definition}\label{Definition: mu cross section}
  Let $\mu$ be an $\{a_t\}$-invariant Borel measure on $X$. We say that $\mathcal{S}\in \mathcal{B}_X$ is a $\mu$-cross-section if there is an $\{a_t\}$-invairant set $X_0\in \mathcal{B}_X$ such that 
    \begin{itemize}
        \item $\mu(X\setminus X_0)=0$;
        \item $\mathcal{S}\cap X_0$ is a Borel cross-section for $(X_0,\mathcal{B}_{X_0})$.
    \end{itemize}
For a $\mu$-cross section $\mathcal{S}$, we define the first return map $\tau_{\mathcal{S}}$ by
\[\tau_{\mathcal{S}}(x):=\begin{cases}
    \tau_{\mathcal{S}\cap X_0}(X), \quad \text{if } x\in \mathcal{S}\cap X_0,\\
    \infty, \quad \quad \text{if   } x\in \mathcal{S}\setminus X_0.
\end{cases}\]
\end{definition}

In the remaining part of this section, we pay attention to some nice properties of a $\mu$-cross-section. Let us make some notations here and hereafter.
Let $\mathcal{S}$ be a $\mu$-cross-section of $X$. For $\epsilon>0$, we define
\[\mathcal{S}_{\geq \epsilon}:=\{x\in \mathcal{S}: \tau_{\mathcal{S}}(x)\geq \epsilon\}, \text{ and }\mathcal{S}_{<\epsilon}=\mathcal{S}\setminus \mathcal{S}_{\geq \epsilon}.\]
Given $E\in \mathcal{B}_X$ and $I\subset \mathbb{R}$, let 
\[E^I:=\{a_t x: x\in E, t\in I\}.\]
We record the following theorem established in \cite{Ambrose_1941_Representation_of_ergodic_flows_MR4730,Ambrose_Kakutani_1942_Structure+and+continuity+of+measurable+flows_MR5800}, which reveals the relation between the measure $\mu$ and the measure $\mu_{\mathcal{S}}$ induced by $\mu$ on a $\mu$-cross-section $\mathcal{S}$.

\begin{theorem}
   Let $\mathcal{S}\in \mathcal{B}_X$. Then for any finite $\{a_t\}$-invariant Borel measure $\mu$ on $X$ for which $\mathcal{S}$ is a $\mu$-cross-section, there exists a $T_{\mathcal{S}}$-invariant measure $\mu_{\mathcal{S}}$ on $(\mathcal{S}, \mathcal{B}_{\mathcal{S}})$ ($\mathcal{B}_{\mathcal{S}}=\{B\cap \mathcal{S}:B\in \mathcal{B}_X\}$) such that the following hold for any $E\in \mathcal{B}_{\mathcal{S}}$:
    \begin{itemize}
        \item[(i)] For any nonempty interval $I$, $\mu_{\mathcal{S}}(E)\geq \frac{\mu(E^I)}{Leb(I)}$, where $Leb$ denotes the usual Lebesgue measure on Euclidean space. In particular, for any $\epsilon>0$, if $E\subset \mathcal{S}_{\geq \epsilon}$, and $I$ is an interval of length less than $\epsilon$, then
        \[\mu_{\mathcal{S}}(E)=\frac{\mu(E^I)}{Leb(I)},\]

      \item[(ii)] In general, 
      \[\mu_{\mathcal{S}}(E)=\lim_{\epsilon\to 0}\frac{\mu(E^{(0,\epsilon)})}{\epsilon}.\]

      \item[(iii)] $\mu_{\mathcal{S}}(E)=0$ if and only if $\mu(E^{\mathbb{R}})=0$.
    \end{itemize}
We call the measure $\mu_{\mathcal{S}}$ obtained above the cross-section measure of $\mu$.
\end{theorem}

\begin{definition}
Let $\mu$ be a finite $\{a_t\}$-invariant Borel measure on $X$, and $\mathcal{S}$ be a $\mu$-cross-section. Let $M\in \mathbb{N}$ and $E\subset \mathcal{S}$ be Borel measurable. We say that $E$ is $M$-tempered if for every point\footnote{It should be a mistake in Definition 4.7 of \cite{Shapira_Weiss_2022_Geometric_and_arithmetic_aspects} of stating that $E$ is tempered if for $\mu_{\mathcal{S}}$-a.e. point $x$ in $E$, the same property holds.} $x\in E$,
\begin{equation}\label{equation: definition of temperedness}
    \# \{t\in [0,1]: a_t x\in E\}\leq M.
\end{equation}
In particular, $E$ is tempered if $E$ is $M$-tempered for some $M\in \mathbb{N}$.
\end{definition}

\begin{definition}\label{Definition: Jordan measurable}
   Let $\mu$ be a Radon measure on $X$. We say that $E\in \mathcal{B}_X$ is Jordan measurable with respect to $\mu$ ($\mu$-JM) if $\mu(\partial_X(E))=0$.
\end{definition}

\begin{definition}\label{Definition: resonable section}
    Let $\mu$ be a probability measure on $X$, and $\mathcal{S}\subset X$ be a $\mu$-cross-section such that $\mathcal{S}$ is locally compact second countable and $\mu_{\mathcal{S}}$ is finite. We say that $\mathcal{S}$ is $\mu$-reasonable if in addition, $\mathcal{S}$ satisfies the following:
    \begin{itemize}
        \item[(1)] For all sufficiently small $\epsilon>0$, the sets $\mathcal{S}_{\geq \epsilon}$ is $\mu_{\mathcal{S}}$-JM.
        \item[(2)] There is a relatively open subset $\mathcal{U}\subset \mathcal{S}$ such that 
        \begin{itemize}
            \item[(a)] The map $(0,1)\times \mathcal{U}\to X$, $(t,x)\mapsto a_t x$ is open;
            \item[(b)] $\mu((cl_X(\mathcal{S})\setminus \mathcal{U})^{(0,1)})=0$.
        \end{itemize}
    \end{itemize}
\end{definition}

\begin{definition}\cite[Definition 5.1]{Shapira_Weiss_2022_Geometric_and_arithmetic_aspects}\label{Definition; genericity}
    Let $\mu$ be a probability measure on $X$ which is invariant under $\{a_t\}$, and $\mathcal{S}$ be a $\mu$-reasonable cross-section.
    \begin{itemize}
        \item[(1)]     We say that a point $x\in X$ is $(a_t,\mu)$-generic if 
    \[\frac{1}{T}\int_0^T \delta_{a_t x} dt\to \mu \text{ in the weak-* topology as }T\to \infty.\]
    Equivalently (see \cite{Billingsley_1968_Convergence+of+probability+measures_MR233396}), $x$ is $(a_t,\mu)$-generic if for any $\mu$-JM set $E$,
    \[\frac{1}{T}\int_0^T \chi_E(a_t x)dt\to \mu(E) \text{ as } T\to \infty.\]

    \item[(2)] For a Borel subset $E\subset \mathcal{S}$ which is $\mu_{\mathcal{S}}$-JM and of positive $\mu_{\mathcal{S}}$-measure, $x\in X$ is $(a_t,\mu_{\mathcal{S}}|_{E})$-generic if the sequence of visits of the orbit $\{a_t x:t>0\}$ to $E$ equidistributes with respect to $\frac{1}{\mu_{\mathcal{S}}(E)}\mu_{\mathcal{S}}|_{E}.$ That is, for any $\mu_{\mathcal{S}}$-JM set $E'\subset E$, we have 
    \[\lim_{n\to \infty}\frac{N(x,T,E')}{N(x,T,E)}=\frac{1}{\mu_{\mathcal{S}}(E)}\mu_{\mathcal{S}}|_{E}(E').\]
    \end{itemize}

\end{definition}

For any measurable $E\subset X$ and $T>0$, consider the number
\[N(x,T,E)=\#\{t\in [0,T]:a_t x\in E\}.\]
The following proposition is the key in the course of establishing Theorem \ref{Theorem: weighted Levy-Khintchine constant}.

\begin{proposition}\label{Proposition: limit behavior of visiting times to tempered set}\cite[Proposition 5.9, Theorem 5.11]{Shapira_Weiss_2022_Geometric_and_arithmetic_aspects}\label{Proposition: transference of genericity to cross section}
    Let $\mathcal{S}$ be a $\mu$-reasonable cross-section and let $E\subset \mathcal{S}$ be a tempered subset which is $\mu_{\mathcal{S}}$-JM. If $x\in X$ is $(a_t,\mu)$-generic, then 
    \[\lim_{T\to\infty} \frac{1}{T} N(x,T,E)=\mu_{\mathcal{S}}(E),\]
and $x$ is $(a_t,\mu_{\mathcal{S}}|E)$-generic.
\end{proposition}

\section{Pointwise equidistribution for almost smooth functions with an error rate}\label{Section: pointwise ergodic theorem}

In this section we will prove Theorem \ref{Theorem: Pointwise ergodic theorem with rate for nonsmooth functions}, which is a strengthening of \cite[Theorem 1.1]{Kleinbock_Shi_Barak_2017_Pointwise_equidistribution_with_an_error_rate_and_with_respect_to_unbounded_functions_MR3606456} and might be of independent interest.

Throughout this section, we let $N=m+n$, $G=SL_{N}(\mathbb{R})$, $\Gamma=SL_N(\mathbb{Z})$, and $\mathfrak{X}_{N}=G/\Gamma$. Denote by $\mu_N$ the unique $G$-invariant probability measure on $\mathfrak{X}_N$. Let $\boldsymbol{a}\in \mathbb{R}^m,\boldsymbol{b}\in \mathbb{R}^n$ be weight vectors. Define the one-parameter diagonal subgroup $\{a_t:t\in \mathbb{R}\}$ of $G$ as in (\ref{equation: definition of a t}).
Recall that $\mathcal{U}$ given by (\ref{align: definition of mathcal U}) denotes a subgroup of upper triangular matrices.

\begin{definition}\label{Definition: almost smooth functions}
    A function $f$ on $\mathfrak{X}_N$ is an almost smooth function if there exist $\{f^-_t\in C_c^{\infty}(\mathfrak{X}_N)\}_{t\in \mathbb{R}_+}$ and $\{f^+_t\in C_c^{\infty}(\mathfrak{X}_N)\}_{t\in \mathbb{R}_+}$  satisfying the following: 
    \begin{itemize}
    \item[(1)] $\sup_t \norm{f^-_t}_{\infty},\sup_t\norm{f^+_t}_{\infty}$ is finite, where $\norm{\cdot}_{\infty}$ is the uniform norm;
    
        \item[(2)] $f^-_t(x)\leq f(x)\leq f^+_t(x)$, $\forall t\in \mathbb{R}_+$ and $\forall x\in \mathfrak{X}_N$;

        \item[(3)] For all sufficiently large $t\in \mathbb{R}_+$, $\mu_N(f^+_t-f^-_t)=O (t^{-1})$;

        \item[(4)]  There exist $M,l\in \mathbb{N}$ such that for all sufficiently large
$t\in \mathbb{R}_+$, 
\[\mathcal{N}_l(f^+_t),\mathcal{N}_l(f^-_t)=O(t^M).\]

    \end{itemize}
\end{definition}

The following are two examples of almost smooth functions but not smooth.

\begin{example}\label{Example: almost smooth function 1}
For $\Lambda\in \mathfrak{X}_N$, let 
    \[\Delta(\Lambda):=\sup_{\boldsymbol{v}\in \Lambda\setminus \{0\}} \log (\frac{1}{\norm{v}}),\]
where $\norm{\cdot}$ is the usual Euclidean norm on $\mathbb{R}^N$.
By Mahler's compactness criterion, for $z\in \mathbb{R}_+$ the set %{This definition of $K_z$ follows from the paper of Kleinbock-Margulis Logarithm laws}
\[ C_z:=\{\Lambda\in \mathfrak{X}_N: \Delta(\Lambda)\leq z\} \]
is compact.
We claim that $f=\chi_{C_z}$ is an almost smooth function.

Indeed, for any $r>0$, let $\theta_r\in C_c^{\infty}(G)$ be the bump function as in Lemma \ref{Lemma: bump function}. Define $\phi_r^-=\theta_r* \chi_{C_{z-r}}$, and $\phi_r^+=\theta_r*\chi_{C_{z+r}}$. For $r>0$ sufficiently small (compared to $z$), we have
\[\chi_{C_{z-2r}}\leq\phi_r^-\leq \chi_{C_z}\leq \phi_r^+\leq \chi_{C_{z+2r}}.\]
Note that
\begin{align*}
    C_{z+2r}\setminus C_{z-2r}\subset \left\{ \Lambda\in \mathfrak{X}_N: \exists \boldsymbol{v}\in \Lambda, e^{-(z+2r)}\leq \norm{\boldsymbol{v}}\leq e^{-(z-2r)}\right\}.
\end{align*}
Let 
\[E_r=\left\{\boldsymbol{v}\in \mathbb{R}^N:  e^{-(z+2r)}\leq \norm{\boldsymbol{v}}\leq e^{-(z-2r)}\right\},\]
and 
\[\hat{\chi}_{E_r}(\Lambda)=\sum_{\boldsymbol{v}\in \Lambda}\chi_{E_r}(\boldsymbol{v}),\quad \forall \Lambda\in \mathfrak{X}_N.\]
Then $\hat{\chi}_{E_r}\in L^1(\mathfrak{X}_N,\mu_N)$ and $\hat{\chi}_{E_r}\geq \chi_{C_{z+2r}\setminus C_{z-2r}}$. Therefore, by Siegal's integration formula, when $r>0$ is sufficiently small, we have
\begin{align*}\label{align: f+ - f-}
    \mu_N( \chi_{C_{z+2r}\setminus C_{z-2r}})&\leq \int \hat{\chi}_{E_r} d\mu_N \ll_{N,z} r.
\end{align*}
Now for any $t>1$, we define 
\begin{align*}
    f_t^-=\phi_{t^{-1}}^-, \text{ and } f_t^+=\phi_{t^{-1}}^+.
\end{align*}
Clearly, $\{f_t^-\}$ and $\{f_t^+\}$ are two sequences of uniformly bounded compactly supported smooth functions, and by the above discussion they satisfy (1)(2)(3) of Definition \ref{Definition: almost smooth functions}. By Lemma \ref{Lemma: bump function}, it is easy to verify that there exist $M,l\in \mathbb{N}$ such that $\mathcal{N}_l(f_t^{\pm})\ll t^M$ for all sufficiently large $t$. The claim follows.
\end{example}
\begin{example}\label{Example: almost smooth funcions 2}
 Let $C$ be an open subset of $\mathfrak{X}_N$ such that $C$ is contained in a compact set and $\partial C$ has measure zero, where $\partial C$ is the boundary of $C$ in $\mathfrak{X}_N$. Then $f=\chi_C$ is almost smooth. Indeed, for any $r>0$ we define the $r$-neighborhood of $\partial C$ by
\begin{align*}
    \partial^r C:=\left\{x\in \mathfrak{X}_N: \dist(x, \partial C)\leq r \right\}.
\end{align*}
Let $\theta_r$ be the bump function as in Lemma \ref{Lemma: bump function}. For $r>0$ sufficiently small (depending on $C$), define 
\begin{align*}
    \phi_r^-=\theta_r*\chi_{C\setminus \partial^{2r} C}, \text{ and } \phi_r^+=\theta_r* \chi_{C \cup \partial^{2r} C}.
\end{align*}
Then $\chi_{C\setminus \partial^{4r}C}\leq \phi_r^- \leq \chi_C \leq \phi_r^+ \leq \chi_{C\cup \partial^{4r}C}$. As $C$ is contained in a compact set, 
\[\mu_N(\chi_{C\cup \partial^{4r}C}-\chi_{C\setminus \partial^{4r}C})\ll_C r.\]
For $t>0$ sufficiently large (depending on $C$), define 
\[f_t^-=\phi_{t^{-1}}^-, \text{ and } f_t^+=\phi_{t^{-1}}^+.\]
It is left to the reader to verify that the sequences $\{f_t^-\}$ and $\{f_t^+\}$ satisfy all assumptions in Definition \ref{Definition: almost smooth functions}.
\end{example}
The following is the main result of this section:

\begin{theorem}\label{Theorem: a general method for proving pointwise ergodic with rate}
    Let $(Y,\nu)$ be a probability space. Let $\{F_t\}_{t\in \mathbb{R}_+}$ be a collection of uniformly bounded measurable functions from $Y\times \mathbb{R}_+$ to $\mathbb{R}$. Suppose that there exist $\delta>0$, $M\geq 0$ such that for any $w\geq t>0$,
\[\left|\int_{Y}F_t(x,t)F_w(x,w)d\nu(x)\right|\ll_{\{F_t\},M} w^M e^{-\delta \min(t,w-t)}.\]
Then given $\epsilon>0$, we have 
\begin{align*}
    \frac{1}{T}\int_0^T F_t(x,t)dt=\begin{cases}
        o(T^{-1/2}\log^{2+\epsilon} T), \quad \text{if } M>0\\
        o(T^{-1/2}\log^{3/2+\epsilon} T), \quad \text{if } M=0
    \end{cases}, \text{ for $\nu$-a.e. $x$.}
\end{align*}
\end{theorem}

\begin{remark}
    Let $F$ be a bounded measurable function from $Y\times \mathbb{R}_+$ to $\mathbb{R}$, and let $F_t=F$ for all $t\in \mathbb{R}_+$. If $M=0$, then Theorem \ref{Theorem: a general method for proving pointwise ergodic with rate} recovers \cite[Theorem 3.1]{Kleinbock_Shi_Barak_2017_Pointwise_equidistribution_with_an_error_rate_and_with_respect_to_unbounded_functions_MR3606456}.
\end{remark}
The proof of Theorem \ref{Theorem: a general method for proving pointwise ergodic with rate} essentially follows from that of \cite[Theorem 3.1]{Kleinbock_Shi_Barak_2017_Pointwise_equidistribution_with_an_error_rate_and_with_respect_to_unbounded_functions_MR3606456}, with more careful estimates. Given $s\in \mathbb{N}$, let $L_s$ be the collection of all intervals of the form $I_j^i=[2^i j, 2^i (j+1)]$, where $i,j\geq 0$ and $2^i(j+1)<2^s$. Denote
\[\mathcal{R}:=\left\{(w,t):w\geq t >0, \text{ and } w^M\geq e^{\delta/2 \min\{t,w-t\}}\right\}.\]
We note that $\mathcal{R}\subset \mathbb{R}^2$ is the region where the integral $\left|\int_{Y}F_t(x,t)F_w(x,w)d\nu(x)\right|$ is potentially large. However, the area of $\mathcal{R}$ is relatively small:

\begin{lemma}\label{Lemmal: area of bad region}
   Let $\{F_t\}_{t\in \mathbb{R}},\delta,M$ be as in the Theorem \ref{Theorem: a general method for proving pointwise ergodic with rate}. Then for any $\epsilon_1>0$, there exists $s(\delta,M,\epsilon_1)\in \mathbb{N}$ such that for any integer $s\geq s(\delta,M,\epsilon_1)$, and $I_j^i\in L_s$, if
   \begin{equation}\label{equation: when i is large}
       \frac{(2+\epsilon_1)\log s}{\log 2}<i<s,
   \end{equation}
   then $Leb(B^i_j\cap \mathcal{R})\ll_{\delta,M} s2^i$, where $B_j^i=I_j^i\times I_j^i$. As a consequence, for $s\geq s(\delta,M,\epsilon_1)$, $I^i_j\in L_s$ with $i$ satisfying (\ref{equation: when i is large}), we have
    \begin{align*}
        \int_Y \left(\int_{I^i_j} F_t(x,t)dt \right)^2 d\nu(x) \ll_{\{F_t\},\delta,M} s 2^i.
    \end{align*}
\end{lemma}
\begin{proof}
  The region $\mathcal{R}$ could be divided into two parts: $\mathcal{R}=\mathcal{R}_1\cup \mathcal{R}_2$, where 
\begin{align*}
 \mathcal{R}_1:=\{(w,t):2t>w\geq t>0, \text{ and } w^M\geq e^{\delta(w-t)/2}\};\\
    \mathcal{R}_2:=\{(w,t):w\geq 2t >0, \text{ and } w^M\geq e^{\delta t/2 }\}.
\end{align*}

First we analyze $\mathcal{R}_1$. In the coordinate plane with vertical axis $t$ and horizontal axis $w$, $\mathcal{R}_1$ is the region in the first quadrant lying below $t=w$, above $t=w/2$ and $t=w-(2M \log w) /\delta$. Let $s(\delta,M,\epsilon_1)$ be a large enough integer such that any $s\geq s(\delta,M,\epsilon_1)$ satisfies the following inequalities:
\begin{equation}\label{equation: inequality about s(delta,l)}
    s>10\cdot \frac{(2+\epsilon_1)\log s}{\log 2}>20\cdot \frac{\log(4M/\delta)+\log\log 2 +\log s}{\log 2} 
\end{equation}
The first half of inequality (\ref{equation: inequality about s(delta,l)}) means that the curve $t=w-(2M\log w)/\delta$ lies above $t=w/2$ as long as $w=2^s$ for $s\geq s(\delta,M,\epsilon_1)$, While the meaning of second half will be clear later in the proof of Lemma \ref{Lemma: upper bound for integral over Ls}. We remark that the choice of $s(\delta,M,\epsilon_1)$ ensures that for any $s\geq s(\delta,M,\epsilon_1)$, there exists at least one solution for $i$ in (\ref{equation: when i is large}).

For a box $B^i_j=I^i_j\times I^i_j$, it is desired that most part of $B_j^i$ lies outside $\mathcal{R}_1$. The worst cases include that $B_j^i$ lies completely above $t=w-(2M\log w)/\delta$, and that the portion of
\begin{equation}\label{equation: bad case 2}
    B_j^i\cap \{(w,t):w/2>\max\{w-(2M\log w)/\delta,0\}\}
\end{equation}
is too large in $B_j^i$. In the following we will show that excluding these two cases, $Leb(B_j^i\cap \mathcal{R}_1)$ is small. We will fix an integer $s\geq s(\delta,M,\epsilon_1)$ for the rest of the proof.

For the first case, it is clear that $B_j^i$ does \textbf{not} completely lie above $t=w-(2M\log w)/\delta$ if and only if the interior of lower side of $B_j^i$ intersects  $t=w-(2M \log w )/\delta$, which happens if and only if for any $w>0$ satisfying $2^i j=w-(2M\log w)/\delta$,
we have $2^i(j+1)>w$. That is, $2^i>(2M\log w)/\delta$. As $0< w\leq 2^s$, if 
\begin{equation}\label{equation: value of i in the first case}
    i>\frac{\log(2M/\delta)+\log\log 2+\log s}{\log 2},
\end{equation}
then $2^i>(2M\log w)/\delta$. As $s$ satisfies (\ref{equation: inequality about s(delta,l)}), when $i$ satisfies inequality (\ref{equation: when i is large}), (\ref{equation: value of i in the first case}) is met and the first case is excluded.

For the second case, let $2^{x_0}=w_0>0$ be such that $w_0=4M\log w_0/\delta$. By inequality (\ref{equation: inequality about s(delta,l)}), we obtain $x_0<s$ if $s\geq s(\delta,M,\epsilon_1)$. Then an elementary computation shows
\[x_0<\frac{\log(4M/\delta)+\log\log 2+\log s}{\log 2}.\] 
Thus, when $i$ satisfies (\ref{equation: when i is large}), by (\ref{equation: inequality about s(delta,l)}) we obtain $w_0<2^{i/2}$. A short calculation shows the Lebesgue measure of (\ref{equation: bad case 2}) is relatively small compared to $B_j^i$ (see (\ref{align: first case when j=0})). In particular, we note that (\ref{equation: bad case 2}) is empty when $i$ satisfies (\ref{equation: when i is large}) and $j\geq 1$.

To conclude, as long as $i$ is a solution of inequality (\ref{equation: when i is large}), the above mentioned two cases are both excluded. Therefore, when $i$ is a solution for (\ref{equation: when i is large}), and $0<j\leq 2^{s-i}-1$, we have
\begin{align*}
    Leb(B_j^i\cap \mathcal{R}_1)&\leq \int_{I_j^i} w-(w-\frac{2M\log w}{\delta}) dw\\
    &\leq \frac{2M}{\delta}\cdot 2^i\log(2^i(j+1))\\
    &\leq \frac{2M\log 2}{\delta}\cdot s2^i,
\end{align*}
When $i$ is as in (\ref{equation: when i is large}) and $j=0$, we have
\begin{align}\label{align: first case when j=0}
     Leb(B_0^i\cap \mathcal{R}_1)&\leq \int_0^{2^{i/2}} \frac{w}{2} dw+ \int_1^{2^i}\frac{2M\log w}{\delta} dw \nonumber\\
     &\leq \log 2\cdot(\frac{1}{4}+\frac{2M}{\delta})\cdot s 2^i.
\end{align}

Now we analyze $\mathcal{R}_2$, which is the region in the first quadrant bounded below by $w$-axis, above by $t=(2M\log w)/\delta$ and $t=w/2$. As before, we let $w_0>0$ be such $w_0=(4M\log w_0)/\delta$. By the above estimation, when $i$ satisfies (\ref{equation: when i is large}). we obtain $w_0<2^{i/2}$.

Therefore, for any $i$ satisfying (\ref{equation: when i is large}), if $0<j\leq 2^{s-i}-1$, then 
\begin{align*}
      Leb(B_j^i\cap \mathcal{R}_2)&\leq \int_{I_i^j} \frac{2M \log w}{\delta} dw\\
      &\leq \frac{2M}{\delta}\cdot \log(2^i(j+1))\int_{I_j^i}dw\\
      &\leq \frac{2M\log 2}{\delta}\cdot s 2^i,
\end{align*}
When $i$ is as in (\ref{equation: when i is large}) and $j=0$, we have
\begin{align*}
    Leb(B_0^i\cap \mathcal{R}_2)&\leq \int_0^{2^{i/2}} \frac{w}{2} dw+ \int_1^{2^i}\frac{2M\log w}{\delta} dw \nonumber\\
     &\leq \log 2\cdot(\frac{1}{4}+\frac{2M}{\delta})\cdot s 2^i.
\end{align*}
To summarize, as long as $i$ satisfies (\ref{equation: when i is large}), for any $0\leq j\leq 2^{s-i}-1$ we have 
\[Leb(B_j^i\cap\mathcal{R})\ll_{\delta,M} s 2^i.\]
This proves the first assertion.

For the second assertion, let $\mathcal{R}'$ be the union of $\mathcal{R}$ and its reflection with respect to $t=w$. If $i$ satisfies (\ref{equation: when i is large}), then for any $0\leq j\leq 2^{s-i}-1$ we have
\begin{align*}
    &\int_Y \left(\int_{2^ij}^{2^i(j+1)} F_t(x,t)dt \right)^2 d\nu(x)\\
    &= \int_{B_j^i}\int_Y F_t(x,t)F_w(x,w)d\nu(x)dtdw \quad \text{ by Fubini as $\{F_t\}$ is uniformly bounded}\\
    &\leq 2\int_{B_j^i\cap \mathcal{R}}\left|\int_{Y} F_t(x,t)F_w(x,w) d\nu(x)\right|dtdw+\int_{B_j^i\setminus\mathcal{R}'}\left|\int_{Y} F_t(x,t)F_w(x,w) d\nu(x)\right|dtdw\\
    &\ll_{\delta,M} 2\sup_t\norm{F_t}_{\infty}^2 \cdot s2^i+\int_{B_j^i}e^{-\delta\min\{t,w-t\}/2}dwdt\\
    &\ll_{\{F_t\},\delta,M} s2^i,
\end{align*}
where the first inequality follows from the symmetry of integral region with respect to $t=w$.
\end{proof}

\begin{lemma}(cf. \cite[Lemma 3.3]{Kleinbock_Shi_Barak_2017_Pointwise_equidistribution_with_an_error_rate_and_with_respect_to_unbounded_functions_MR3606456})\label{Lemma: upper bound for integral over Ls}
Given $\epsilon_1>0$. Let $s(\delta,M,\epsilon_1)\in \mathbb{N}$ be as in Lemma \ref{Lemmal: area of bad region}. Then for any $s\geq s(\delta,M,\epsilon_1)$,
    \begin{align*}
        \sum_{I\in L_s}\int_Y \left(\int_I F_t(x,t)dt\right)^2  d\nu(x) \ll_{\{F_t\},\delta,M} s^{2+\epsilon_1} 2^s.
    \end{align*}
\end{lemma}
\begin{proof}
Let $s\geq s(\delta,M,\epsilon_1)$, define
\begin{align*}
    L_s^1=\left\{[2^ij,2^i(j+1)]\in L_s: i\leq \lfloor \frac{(2+\epsilon_1)\log s}{\log 2} \rfloor +1\right\}; \quad
    L_s^2=L_s\setminus L_s^1.
\end{align*}
Setting $b=\lfloor \frac{(2+\epsilon_1)\log s}{\log 2} \rfloor +1$, we have
\begin{align*}
    \sum_{I\in L_s^1} \int_Y \left(\int_I F_t(x,t)dt\right)^2  d\nu(x)&\leq \sum_{i=0}^{b} \sum_{j=0}^{2^{s-i}-1}  \int_Y \left(\int_{I^i_j} F_t(x,t)dt\right)^2  d\nu(x)\\
    &\leq \sum_{i=0}^{b} \sum_{j=0}^{2^{s-i}-1} 2^{2i} \sup_t\norm{F_t}_{\infty}^2\\
    &\leq 2^{s+b+1} \sup_t\norm{F_t}_{\infty}^2\\
    &\ll_{\{F_t\},\delta,M} s^{2+\epsilon_1}2^s. 
\end{align*}
For $L_s^2$, we have
\begin{align*}
    \sum_{I\in L_s^2} \int_Y \left(\int_{I} F_t(x,t)dt\right)^2  d\nu(x)&\leq \sum_{i=b+1}^{s-1} \sum_{j=0}^{2^{s-i}-1}  \int_Y \left(\int_{I^i_j} F_t(x,t)dt\right)^2  d\nu(x)\\
    &\ll_{\{F_t\},\delta,M} \sum_{i=0}^{s-1}\sum_{j=0}^{2^{s-i}-1} s2^i\\
    &\leq s^2 2^s,
\end{align*}
where the second inequality follows from Lemma \ref{Lemmal: area of bad region}. Combining the above estimates, the lemma is proven.
\end{proof}

\begin{lemma}(cf. \cite[Lemma 3.5]{Kleinbock_Shi_Barak_2017_Pointwise_equidistribution_with_an_error_rate_and_with_respect_to_unbounded_functions_MR3606456})\label{Lemma: a small set over which integral is large}
    For any $\epsilon_1>0$, let $s(\delta,M,\epsilon_1)$ be as in Lemma \ref{Lemmal: area of bad region}. Given $\epsilon_2>0$, for any $s\in \mathbb{N}$, we define
\[Y_s(\epsilon_1,\epsilon_2):=\left\{ y\in Y:  \sum_{I\in L_s}\left(\int_I F_t(x,t)dt\right)^2 > s^{3+\epsilon_1+2\epsilon_2} 2^s   \right\}.\]
Then
    \begin{itemize}
        \item[(1)] For any $s\geq s(\delta,M,\epsilon_1)$,
    \[\nu(Y_s(\epsilon_1,\epsilon_2))\ll_{\{F_t\},\delta,M} s^{-1-2\epsilon_2};\]

    \item[(2)] For $s\in\mathbb{N}$, any positive integer $k< 2^s$, and $x\notin Y_s(\epsilon_1,\epsilon_2)$,
    \[\left| \int_0^k F_t(x,t)dt \right| \leq s^{2+\epsilon_1/2+\epsilon_2}\cdot 2^{s/2}.   \]
    \end{itemize}
  
\end{lemma}
\begin{proof}
Note that as $s\geq s(\delta,M,\epsilon_1)$, by Lemma \ref{Lemma: upper bound for integral over Ls} we have
\begin{align*}
    \sum_{I\in L_s}\int_{Y_s(\epsilon_1,\epsilon_2)}\left(\int_I F_t(x,t)dt\right)^2 d\nu(x)&\leq   \sum_{I\in L_s}\int_{Y}\left(\int_I F_t(x,t)dt\right)^2 d\nu(x)\\
    &\ll_{\{F_t\},\delta,M} s^{2+\epsilon_1} 2^s.
\end{align*}
On the other hand, by definition of $Y_s(\epsilon_1,\epsilon_2)$,
\begin{align*}
    \sum_{I\in L_s}\int_{Y_s(\epsilon_1,\epsilon_2)}\left(\int_I F_t(x,t)dt\right)^2 d\nu(x)&=\int_{Y_s(\epsilon_1,\epsilon_2)}\sum_{I\in L_s}\left(\int_I F_t(x,t)dt\right)^2 d\nu(x)\\
    &\geq s^{3+\epsilon_1+2\epsilon_2} 2^s \nu(Y_s(\epsilon_1,\epsilon_2)).
\end{align*}
Therefore,
\[\nu(Y_s(\epsilon_1,\epsilon_2))\ll_{\{F_t\},\delta,M} s^{-1-2\epsilon_2}.\]
This proves $(1)$.

For every $k\in \mathbb{N}$ with $k<2^s$, $[0,k]$ can be covered by at most $s$ nonoverlapping intervals in $L_s$ by \cite[Lemma 3.4]{Kleinbock_Shi_Barak_2017_Pointwise_equidistribution_with_an_error_rate_and_with_respect_to_unbounded_functions_MR3606456}. 
We let $L(k)\subset L_s$ be the collection of these intervals. For $x\notin Y_s(\epsilon_1,\epsilon_2)$. we have
\begin{align*}
    \left(\int_0^k F_t(x,t)dt \right)^2&\leq\left( \sum_{I\in L(k)}\left|\int_I F_t(x,t)dt\right|\right)^2\\
    &\leq s \cdot \sum_{I\in L(k)}\left(\int_I F_t(x,t)dt\right)^2 \quad \text{ by Cauchy-Schwartz}\\
    &\leq s^{4+\epsilon_1+2\epsilon_2} 2^s,
\end{align*}
This proves $(2)$.
\end{proof}

\begin{proof}[Proof of Theorem \ref{Theorem: a general method for proving pointwise ergodic with rate}]

We only give a proof when $M>0$, as when $M=0$, the theorem follows from \cite[Theorem 3.1]{Kleinbock_Shi_Barak_2017_Pointwise_equidistribution_with_an_error_rate_and_with_respect_to_unbounded_functions_MR3606456}.

In the following we assume $M>0$. Given $\epsilon>0$, we choose $\epsilon_1,\epsilon_2>0$ such that $\epsilon_1+\epsilon_2<\epsilon/2$. Define $Y_s(\epsilon_1,\epsilon_2)$ as in Lemma \ref{Lemma: a small set over which integral is large} for $s\in \mathbb{N}$. By Lemma \ref{Lemma: a small set over which integral is large} (1), as $s(\delta,M,\epsilon_1)$ is finite, we have
\[\sum_{s=1}^{\infty}\nu(Y_s(\epsilon_1,\epsilon_2))<\infty.\]
By Borel-Cantelli Lemma, there exists $Y(\epsilon_1,\epsilon_2)\subset Y$ such that $\nu(Y(\epsilon_1,\epsilon_2))=1$, and for any $y\in Y(\epsilon_1,\epsilon_2)$, there exists $s_y\in \mathbb{N}$ such that for all $s\geq s_y$, $y\notin Y_s(\epsilon_1,\epsilon_2)$.

Let $y\in Y(\epsilon_1,\epsilon_2)$. Suppose $T\geq \min\{2^{s_y-1},2\}$. Let $k=\lfloor T \rfloor$, $s=\lfloor 1+\log_2 k \rfloor$, then $2^{s-1}\leq k\leq T\leq k+1\leq 2^s$. We have by Lemma \ref{Lemma: a small set over which integral is large} (2),
\begin{align*}
    \left|\int_0^T F_t(y,t)dt \right|&\leq \left| \int_k^T F_t(y,t)dt\right|+\left|\int_0^k F_t(y,t)dt \right|\\
    &\leq \sup_t \norm{F_t}_{\infty}+s^{2+\epsilon_1/2+\epsilon_2} 2^{s/2}\\
    &\leq \sup_t \norm{F_t}_{\infty}+(2T)^{1/2} 2^{2+\epsilon_1/2+\epsilon_2} (\frac{\log T}{\log 2})^{2+\epsilon_1/2+\epsilon_2}\\
    &\leq \sup_t \norm{F_t}_{\infty}+8 T^{1/2} (\frac{\log T}{\log 2})^{2+\epsilon/2}.
\end{align*}
As $\epsilon>0$ is arbitrary, the theorem follows.

\end{proof}

\begin{proof}[Proof of Theorem \ref{Theorem: Pointwise ergodic theorem with rate for nonsmooth functions}]

Let $f$ be an almost smooth function, and $\{f_t^-\}$, $\{f_t^+\}$ be the corresponding sequences of compactly supported smooth functions as in Definition \ref{Definition: almost smooth functions}. Let $\phi\in C_c^{\infty}(\mathcal{U})$ such that $\phi\geq 0$ and $\int_{\mathcal{U}}\phi dA=1$. Define $F_t^-=f_t^- -\mu_N(f_t^-)$.  Given any $\Lambda\in \mathfrak{X}_N$, by Theorem \ref{Theorem: effective mixing} and assumptions of Definition \ref{Definition: almost smooth functions}, there exist $\delta>0, l\in \mathbb{N}$ such that for $w\geq t\geq 0$,
\begin{align*}
    \int_{\mathcal{U}} F_w^-(a_w u(A)\Lambda) F_t^-(a_t u(A)\Lambda)\phi(A)dA=O(e^{-\delta\min\{t,w-t\}}) w^M),
\end{align*}
where we omit the dependence of the constant on $\Lambda,\phi$ as they are fixed.
We note that same estimate holds for functions $F_t^+=f_t^+-\mu_N(f_t^+)$.

Define $d\nu(A)=\phi(A)dA$, then $\nu$ is a probability measure on $\mathcal{U}$ by the choice of $\phi$. Applying Theorem \ref{Theorem: a general method for proving pointwise ergodic with rate} to the probability space $(\mathcal{U},\nu)$, the collection of uniformly smooth functions $\{F_t\}_{t\in\mathbb{R}_+}$ with $F_t(A,s)=F^-_t(a_s u(A)\Lambda)$ (and $F_t(A,s)=F^+_t(a_s u(A)\Lambda)$), we obtain a set $\mathcal{U}_{\nu}$ with full $\nu$-measure such that any $A\in \mathcal{U}_{\nu}$ satisfies the conclusion of Theorem \ref{Theorem: a general method for proving pointwise ergodic with rate} for $\{F_t^+\}$ and $\{F_t^-\}$.

Therefore, given any $A\in \mathcal{U}_{\nu}$, for any $T>0$, we have
\begin{align*}
    &\frac{1}{T}\left|\int_0^T f(a_t u(A)\Lambda)-\mu_N(f)dt \right|\\
    &\leq\frac{1}{T} \left(\left|\int_0^T f(a_t u(A)\Lambda)-f_t^-(a_t u(A)\Lambda)dt\right|+\left|\int_0^T f_t^-(a_t u(A)\Lambda)-\mu_N(f_t^-)dt\right|\right)\\
    &+\frac{1}{T}\left|\int_0^T \mu_N(f_t^-)-\mu_N(f)dt \right|\\
    &= I_1+I_2+I_3,
\end{align*}
where
\begin{align*}
  I_1&=\frac{1}{T}\int_0^T  f(a_t u(A)\Lambda)-f_t^-(a_t u(A)\Lambda)dt \\
  &\leq \frac{1}{T}\int_0^T f_t^+(a_t u(A)\Lambda)-f_t^-(a_t u(A)\Lambda)dt\\
  &=\frac{1}{T}\int_0^T f_t^+(a_t u(A)\Lambda)-\mu_N(f_t^+)+\mu_N(f_t^+)-\mu_N(f_t^-)+\mu_N(f_t^-)-f_t^-(a_t u(A)\Lambda)dt\\
  &\leq o(T^{-1/2}\log^{2+\epsilon}T)+\frac{1}{T}\int_1^T t^{-1}dt\\
  &=o(T^{-1/2}\log^{2+\epsilon}T),
\end{align*}
By the choice of $A$, we have
\begin{align*}
    I_2=\frac{1}{T}\left|\int_0^T f_t^-(a_t u(A)\Lambda)-\mu_N(f_t^-)dt\right|=o(T^{-1/2}\log^{2+\epsilon}T).
\end{align*}
Finally, by assumptions of Definition \ref{Definition: almost smooth functions},
\begin{align*}
    I_3&\leq \frac{1}{T}\int_0^T \mu_N(f_t^+)-\mu_N(f_t^{-})dt\ll \frac{1}{T}\int_0^T t^{-1}dt\\
    &=o(T^{-1/2}\log^{2+\epsilon}T).
\end{align*}
Combining the above estimates, we conclude that for any $A\in \mathcal{U}_{\nu}$,
\[\frac{1}{T}\int_0^T f(a_t u(A)\Lambda)dt=\mu_N(f)+o(T^{-1/2}\log^{2+\epsilon}T).\]
To finish the proof, for any $\phi\in C_c^{\infty}(\mathcal{U})$, consider the set 
\[S_{\phi}:=\{A\in \mathcal{U}: \phi(A)>0\}.\]
Since countably many sets of the form $S_{\phi}$ exhaust $\mathcal{U}$, the theorem is proved.

\end{proof}

\section{Weighted best approximation vectors}\label{Section: cross sections}
In this section, we recall the theory of best approximation in weighted case, and then identify certain subset of a cross section in the space of unimodular lattices
 corresponding to weighted best approximation. In this section and the next, let $d\geq 1$ be an integer. 

Let $\boldsymbol{w}={^t(}w_1,\cdots,w_d)\in \mathbb{R}^d$ be a weight vector. Without loss of generality, for a weighted vector $\boldsymbol{w}$, we may assume that $w_1\geq w_2\geq \cdots\geq w_d>0$. We define a $\boldsymbol{w}$-quasi-norm $\norm{\cdot}_{\boldsymbol{w}}$ as in (\ref{equation: definition of w-quasi norm}).
Although $\norm{\cdot}_{\boldsymbol{w}}$ is not a norm, by convexity of the map $s\mapsto s^{1/w_i}$ it satisfies the inequality
\begin{equation}\label{equation: triangle inequality for quasi norm}
\norm{\boldsymbol{x}+\boldsymbol{y}}_{\boldsymbol{w}} \leq 2^{\frac{1-w_d}{w_d}} (\norm{\boldsymbol{x}}_{\boldsymbol{w}} +\norm{\boldsymbol{y}}_{\boldsymbol{w}}) \ \ \\ {\rm for\  all\ \boldsymbol{x},\boldsymbol{y}\in \mathbb{R}^d .}
\end{equation}

\subsection{Weighted best approximable vectors and the corresponding cross-section}
We recall the following definition of $\boldsymbol{w}$-weighted best approximation vectors for $\boldsymbol{\theta}\in \mathbb{R}^d$.

\begin{definition}\label{Definition: w best approximation vector}
    $(\boldsymbol{p},q)\in \mathbb{Z}^d\times \mathbb{N}$ is a $\boldsymbol{w}$-best approximation vector for $\boldsymbol{\theta}\in \mathbb{R}^d$ if
\begin{itemize}
    \item[(i)] $\norm{q\boldsymbol{\theta}-\boldsymbol{p}}_{\boldsymbol{w}}<\norm{q'\boldsymbol{\theta}-\boldsymbol{p}'}_{\boldsymbol{w}}$ for any $(\boldsymbol{p}',q')\in \mathbb{Z}^d\times \mathbb{N}$ with $q^{\prime}<q$;

    \item[(ii)]   $\norm{q\boldsymbol{\theta}-\boldsymbol{p}}_{\boldsymbol{w}}\leq\norm{q\boldsymbol{\theta}-\boldsymbol{p}'}_{\boldsymbol{w}}$ for any $\boldsymbol{p}'\in \mathbb{Z}^d$.
    
\end{itemize}
\end{definition}
Let $G=SL_{d+1}({\mathbb{R}})$, $\Gamma=SL_{d+1}(\mathbb{Z})$, and $\mathfrak{X}_{d+1}=G/\Gamma$. For any $\boldsymbol{\theta}={^t(}\theta_1,\cdots,\theta_d)\in \mathbb{R}^d$, we consider the following upper triangular matrix associated to $\boldsymbol{\theta}$:
\begin{align*}
u(\boldsymbol{\theta}):=\begin{pmatrix}
        1& & \theta_1\\
         & \ddots & \vdots\\
         & & \theta_d\\
         & & 1
    \end{pmatrix}
    \in G.
\end{align*}
Denote $\Lambda_{\boldsymbol{\theta}}=u(\boldsymbol{-\theta})\mathbb{Z}^{d+1}\in \mathfrak{X}_{d+1}$. 
We also consider the weighted one-parameter diagonal flow $\{a_t: t\in \mathbb{R}\}\subset G$ defined by 
\begin{equation}\label{equation: definition of at from w}
    a_t:=\diag({e^{w_1 t},\cdots,e^{w_d t},e^{-t}}).
\end{equation}

To describe the cross-section associated to weighted best approximation, we introduce following notations in \cite{Shapira_Weiss_2022_Geometric_and_arithmetic_aspects}. For any subset $W\subset \mathbb{R}^{d+1}$, any positive integer $k$, we denote 
\[\mathfrak{X}_{d+1}(W,k):=\{\Lambda \in \mathfrak{X}_{d+1}: \# (\Lambda_{prim}\cap W) \geq k\},\]
where $\Lambda_{prim}$ is the collection of all nontrivial primitive vectors\footnote{A nonzero vector $\boldsymbol{v}\in \Lambda$ is primitive if $\mathbb{R} \boldsymbol{v}\cap \Lambda = \mathbb{Z} \boldsymbol{v}$.} in $\Lambda$. For $k=1$, we denote
\[\mathfrak{X}_{d+1}(W):=\mathfrak{X}_{d+1}(W,1).\]
For $r>0$ we define the following subsets of $\mathbb{R}^{d+1}$:
\begin{align*}
    D_r:=\{\boldsymbol{x}\in \mathbb{R}^{d+1}: \norm{\pi_{\mathbb{R}^d}(\boldsymbol{x})}_{\boldsymbol{w}}\leq r, \text{ and } x_{d+1}=1\};\\
    C_r:=\{\boldsymbol{x}\in \mathbb{R}^{d+1}: \norm{\pi_{\mathbb{R}^d}(\boldsymbol{x})}_{\boldsymbol{w}}\leq r, \text{ and } |x_{d+1}|\leq 1\},
\end{align*}
where for $\boldsymbol{x}={^t(}x_1,\cdots,x_{d+1})\in \mathbb{R}^{d+1}$, $\pi_{\mathbb{R}^d}(\boldsymbol{x})={^t(}x_1,\cdots,x_d)$. 
The disk $D_r$ and the cylinder $C_r$ yield subsets of $\mathfrak{X}_{d+1}$ defined by:
\begin{align*}
    \mathcal{S}_r=\mathfrak{X}_{d+1}(D_r):=\{\Lambda\in \mathfrak{X}_{d+1}:\#(\Lambda_{prim}\cap D_r)\geq 1\};\\
    \mathcal{S}_r^{\#}=\mathfrak{X}_{d+1}^{\#}(D_r):=\{\Lambda\in \mathfrak{X}_{d+1}:\#(\Lambda_{prim}\cap D_r)= 1\},
\end{align*}
We will show shortly that $\mathcal{S}_r$ is a $\mu_{d+1}$-cross-section, where $\mu_{d+1}$ is the $G$-invariant probability measure on $\mathfrak{X}_{d+1}$. From now on, for any $\Lambda\in \mathcal{S}_r^{\#}$, we denote by
\begin{equation}\label{equation: definition of v(lambda)}
    \{\boldsymbol{v}(\Lambda)\}=\Lambda_{prim}\cap D_r
\end{equation}
the unique primitive vector of $\Lambda$ in $D_r$, and 
\begin{equation}\label{equation: definition of r(lambda)}
    r(\Lambda)=\norm{\boldsymbol{v}(\Lambda)}_{\boldsymbol{w}}.
\end{equation}
When $r=1$, a subset $\mathcal{B}(\boldsymbol{w})$ of $\mathcal{S}_1^{\#}$ is of special interest:
\begin{align}\label{align: definition of open set B in the cross section}
    \mathcal{B}(\boldsymbol{w}):=\{\Lambda\in \mathcal{S}_1^{\#}: C_{r(\Lambda)}\cap \Lambda_{prim}=\{\pm \boldsymbol{v}(\Lambda)\}\}.
\end{align}
From now on in the sequel, since the weight vector $\boldsymbol{w}$ is fixed, we will write $\mathcal{B}(\boldsymbol{w})=\mathcal{B}$.

\begin{proposition}\label{Proposition: dynamical formulation of best approximable vectors}
Let $\boldsymbol{\theta}\in \mathbb{R}^d$. If for some $t> 0$, $a_t\Lambda_{\boldsymbol{\theta}}\in \mathcal{B}$, then $^t(\boldsymbol{p},q)=u(\boldsymbol{\theta})a_{-t}\boldsymbol{v}(a_t\Lambda_{\boldsymbol{\theta}})\in \mathbb{Z}^{d+1}_{prim}$ is a $\boldsymbol{w}$-best approximable vector of $\boldsymbol{\theta}$. Conversely, if $^t(\boldsymbol{p},q)$ is a $\boldsymbol{w}$-best approximable vector of $\boldsymbol{\theta}$ with $q>2^{1/w_d}$, then for $t=\log q$, $a_t\Lambda_{\boldsymbol{\theta}}\in \mathcal{B}$. 
\end{proposition}

\begin{proof}
Assume that for some $t> 0$, $a_t\Lambda_{\boldsymbol{\theta}}\in \mathcal{B}$, then clearly $t=\log q$. We want to verify that $^t(\boldsymbol{p},q)=u(\boldsymbol{\theta})a_{-t}\boldsymbol{v}(a_t\Lambda_{\boldsymbol{\theta}})$ satisfies (i)(ii) of Definition \ref{Definition: w best approximation vector}.

Let $\boldsymbol{v}'={^t(}\boldsymbol{p}^{\prime},q^{\prime})\in \mathbb{Z}^d\times \mathbb{N}$ be such that $q^{\prime}<q$, then by assumption we have $a_t u(-\boldsymbol{\theta})\boldsymbol{v}^{\prime} \notin C_{r}$ for $r=r(a_t \Lambda_{\boldsymbol{\theta}})=q\norm{q\boldsymbol{\theta}-\boldsymbol{p}}_{\boldsymbol{w}}$. Note that
\begin{align*}
    a_t u(-\boldsymbol{\theta}) \boldsymbol{v}'= {^t(q^{w_1}(q' \theta_1-p_1'),\cdots,q^{w_d}(q'\theta_d-p_d'), q'/q)}.
\end{align*}
As $q^{\prime}/q<1$, this implies that $\norm{\pi_{\mathbb{R}^d}(a_t u(-\boldsymbol{\theta})\boldsymbol{v}')}_{\boldsymbol{w}}=q\norm{q'\boldsymbol{\theta}-p'}_{\boldsymbol{w}}>r$. Thus $\norm{q\boldsymbol{\theta}-\boldsymbol{p}}_{\boldsymbol{w}}<\norm{q'\boldsymbol{\theta}-\boldsymbol{p}'}_{\boldsymbol{w}}$, and (i) in Definition \ref{Definition: w best approximation vector} is satisfied. (ii) is verified using the same argument as above.

Conversely, assume $^t(\boldsymbol{p},q)$ is a $\boldsymbol{w}$-best approximable vector of $\boldsymbol{\theta}$ with $q>2^{{1}/{w_d}}$. Let $r=q\norm{q\boldsymbol{\theta}-\boldsymbol{p}}_{\boldsymbol{w}}$. For $t=\log q$, we claim that $a_t \Lambda_{\boldsymbol{\theta}}\in \mathcal{B}$. Indeed, for any $^t(\boldsymbol{p}',q')\in \mathbb{Z}^d\times \mathbb{N}$ with $q'<q$, we have $\norm{q\boldsymbol{\theta}-\boldsymbol{p}}_{\boldsymbol{w}}<\norm{q'\boldsymbol{\theta}-\boldsymbol{p}'}_{\boldsymbol{w}}$, which implies that $a_t u(-\boldsymbol{\theta}){^t(}\boldsymbol{p}',q')\notin C_r$.

Moreover, the volume of $C_r$ is $2^{d+1}r$. As $C_r$ is a symmetric convex body and its interior contains no nonzero vector of $a_t \Lambda_{\boldsymbol{\theta}}$, Minkowski's first theorem implies that $r=q\norm{q\boldsymbol{\theta}-\boldsymbol{p}}_{\boldsymbol{w}}\leq 1$. As $q> 2^{1/w_d}$, we have $|q\theta_i-p_i|<1/2$ for any $1\leq i\leq d$. Therefore, for any $\boldsymbol{p}'\in \mathbb{Z}^d$ such that $\boldsymbol{p}'\neq \boldsymbol{p}$, we have $|q\theta_i-p_i'|\geq 1/2$ for some $1\leq i \leq d$, and so $\norm{q\boldsymbol{\theta}-\boldsymbol{p}'}_{\boldsymbol{w}}\geq (1/2)^{1/w_d}>\norm{q\boldsymbol{\theta}-\boldsymbol{p}}_{\boldsymbol{w}}$. This again implies that $a_t u(-\boldsymbol{\theta}){^t(}\boldsymbol{p}',q)\notin C_r$ for any $\boldsymbol{p}'\neq \boldsymbol{p}$. Therefore, $a_t\Lambda_{\boldsymbol{\theta}}\in \mathcal{B}$.

\end{proof}

\begin{definition}\cite[Definition 10.3]{Shapira_Weiss_2022_Geometric_and_arithmetic_aspects}\label{Definition: Prefix equivalent sequence}\label{Definition: prefix equivalence} Two infinite sequences $(a_k)_{k=1}^{\infty}, (b_k)_{k=1}^{\infty}$ of $\mathbb{R}$ are prefix-equivalent if there exist $k_0,l_0\in \mathbb{N}$ such that $a_{k_0+i}=b_{l_0+i}$ for all $i\geq 0$.
\end{definition}
An immediate consequence of Proposition \ref{Proposition: dynamical formulation of best approximable vectors} is the following:

\begin{corollary}\label{Corollary: two sequences are prefix equivalent}
Let $\boldsymbol{\theta}\in \mathbb{R}^d$ be such that the set
\[\{t>0:a_t \Lambda_{\boldsymbol{\theta}}\in \mathcal{B}\}=\{t_1<t_2<\cdots\}\]
is infinite\footnote{We will see shortly there are plenty of such $\boldsymbol{\theta}'s$.}. Let $\{{^t(}\boldsymbol{p}'_n,q'_n)\in \mathbb{Z}^d\times \mathbb{N}: n\in \mathbb{N}\}$ be the sequence obtained by 
\[^t(\boldsymbol{p}'_n,q'_n)=u(\boldsymbol{\theta})a_{-t_n}\boldsymbol{v}(a_t\Lambda_{\boldsymbol{\theta}})\in \mathbb{Z}^{d+1},\quad \forall n\in \mathbb{N}.\]
On the other hand, let $\{{^t(}\boldsymbol{p}_n,q_n)\in \mathbb{Z}^d\times \mathbb{N}:n\in \mathbb{N}\}$ be the sequence of all $\boldsymbol{w}$-best approximable vectors of $\boldsymbol{\theta}$. Then $\{q'_n\}_{n\in \mathbb{N}}$ and $\{q_n\}_{n\in\mathbb{N}}$ are prefix-equivalent.
\end{corollary}

\begin{proof}
    We may take $k_0$ to be the smallest positive integer such that $t_{k_0}>\log 2/w_d$. Then by Proposition \ref{Proposition: dynamical formulation of best approximable vectors}, we can choose $l_0$ to be the unique integer such that $q'_{k_0}=q_{l_0}$. Applying Proposition \ref{Proposition: dynamical formulation of best approximable vectors} again, we obtain
    \[^t(\boldsymbol{p}'_{k_0+i},q'_{k_0+i})={^t(}\boldsymbol{p}_{l_0+i},q_{l_0+i}),\quad \forall i\geq 0.\]
\end{proof}

\subsection{Cross-section is reasonable}
The argument provided in the following Lemma is a slight improvement of that in the proof of \cite[Lemma 8.4]{Shapira_Weiss_2022_Geometric_and_arithmetic_aspects}.
\begin{lemma}\label{Lemma: the visit times are unbounded}
Let 
\begin{align*}
    &X_h:=\left\{\Lambda\in \mathfrak{X}_{d+1}: \Lambda \cap {^t(}\boldsymbol{0},\mathbb{R})=\{\boldsymbol{0}\}\right\}, \text{ and }\\
    &X_v:=\left\{\Lambda\in \mathfrak{X}_{d+1}: \Lambda \cap {^t(}\mathbb{R}^d,0)=\{\boldsymbol{0}\}\right\}
\end{align*}
Then for any $\Lambda\in X_h$\footnote{Here the subscripts $h$ and $v$ stand for horizontal and vertical components, respectively. For we view the first $d$ rows of a vector in $\mathbb{R}^{d+1}$ as its horizontal component, and the last row as its vertical component.}, the set of visit times $\mathcal{Y}_{\Lambda}$ is unbounded from above; For any $\Lambda\in X_v$, $\mathcal{Y}_{\Lambda}$ is unbounded from below.
\end{lemma}
\begin{proof}
    To show that $\mathcal{Y}_{\Lambda}$ is unbounded from above for any $\Lambda \in X_h$, we proceed by contradiction. Assume there exists $\Lambda\in X_h$ such that $\mathcal{Y}_{\Lambda}$ has an upper bound $T>0$. Consider the set
\[C_1^{[-T,0]}:=\bigcup_{t\in [-T,0]} a_t C_1.\]
By compactness,
$ C_1^{[-T,0]}\cap \Lambda\setminus\{\boldsymbol{0}\}=\{\boldsymbol{v}_1,\cdots, \boldsymbol{v}_k\}$ is finite ($k$ might be 0, in this case this is an empty set). By definition of $X_h$, any nonzero vector $\boldsymbol{w}\in \Lambda$ satisfies $\pi_{\mathbb{R}^d}(\boldsymbol{w})\neq \boldsymbol{0}$. Thus, we can choose $t\geq T$ large enough so that $a_t \boldsymbol{v}_i \notin C_1$ for all $1\leq i\leq k$. On the other hand, by Minkowski's first theorem, $a_t \Lambda \cap C_1\setminus\{\boldsymbol{0}\}$ is nonempty. Let $\boldsymbol{v}'\in \Lambda\setminus\{\boldsymbol{0}\}$ be such that $a_t\boldsymbol{v}' \in C_1$. Then $\boldsymbol{v}'\neq \boldsymbol{v}_i$ for all $1\leq i\leq k$. In particular, $\boldsymbol{v}'\notin C_1$. Moreover, if we write $\boldsymbol{v}'={^t(}v_1',\cdots,v_{d+1}')$, then $v_{d+1}'\neq 0$. This is because $t>0$, applying $a_t$ to $\boldsymbol{v}'$ will expand $|v_1'|,\cdots,|v_d'|$. If $v_{d+1}'$ were zero, then in order that $a_t \boldsymbol{v}'\in C_1$, it must be that $\boldsymbol{v}'\in C_1$, contradicting the choice of $\boldsymbol{v}'$.

Without loss of generality, we may assume $v_{d+1}'>0$. Then as $\boldsymbol{v}'\notin C_1$, yet $a_t \boldsymbol{v}'\in C_1$, we conclude that the last row of $\boldsymbol{v}'$ is greater than $1$, and so there exists $0< s <t$ such that $a_s \boldsymbol{v}'\in D_1$. By assumption, we also have $s\leq T$, and so $\boldsymbol{v}'\in a_{-s} D_1\subset C_1^{[-T,0]}$. This leads to a contradiction, and finishes the proof of the first assertion.

To show that $\mathcal{Y}_{\Lambda}$ is unbounded from below for $\Lambda\in X_v$, we again assume by contradiction that $\mathcal{Y}_{\Lambda}$ has a lower bound $-T<0$ for some $T>0$ and some $\Lambda\in X_v$. Then we consider the set $C_1^{[0,T]}$. Running a similar argument as we did for elements in $X_h$ would yield a contradiction.
\end{proof}

\begin{lemma}\cite[Lemma 8.4]{Shapira_Weiss_2022_Geometric_and_arithmetic_aspects}
    Let $\mu$ be any $\{a_t\}$-invariant probability measure on $\mathfrak{X}_{d+1}$. Then $\mathcal{S}_{1}$ is a $\mu$-cross-section for $(\mathfrak{X}_{d+1},\mu, \{a_t\})$. Furthermore, the cross-section measure satisfies 
    \[ \mu_{S_1}(\mathfrak{X}_{d+1}(D_1,2))=0.    \]
\end{lemma}

\begin{proof}
The proof provided here resembles that in \cite[Lemma 8.4]{Shapira_Weiss_2022_Geometric_and_arithmetic_aspects}. Let 
\[X_0=\left\{\Lambda\in \mathfrak{X}_{d+1}: \Lambda \cap {^t(}\boldsymbol{0},\mathbb{R})=\{\boldsymbol{0}\}, \text{ and }\Lambda \cap {^t(}\mathbb{R}^d,0)=\{\boldsymbol{0}\}\right\}.\]
Note that $X_0=X_h\cap X_v$, where $X_h,X_v$ are as in Lemma \ref{Lemma: the visit times are unbounded}. By \cite[Lemma 8.2]{Shapira_Weiss_2022_Geometric_and_arithmetic_aspects}, $X_0$ is a $G_{\delta}$-set, thus measurable, and $(X_0, \mathcal{B}_{X_0})$ is a standard Borel space, where $\mathcal{B}_{X_0}$ is the Borel $\sigma$-algebra of $X_0$. To show that $\mathcal{S}_1$ is a $\mu$-cross section, we need to verify (i)(ii)(iii) of Definition \ref{Definition: mu cross section}.

For (i), since $\mu(\mathfrak{X}_{d+1})=1$, there exists a sequence of compact sets $\{K_n\}_{n\in \mathbb{N}}$ exhausting $\mathfrak{X}_{d+1}$ with respect to $\mu$. So for any $\epsilon>0$, there exists $N\in \mathbb{N}$ such that for any $n\geq N$, $\mu(K_n)>1-\epsilon$. Suppose $\mu(\mathfrak{X}_{d+1}\setminus X_0)>0$, then there exists $N\in \mathbb{N}$ such that for any $n\geq N$, $\mu(K_n\cap (\mathfrak{X}_{d+1}\setminus X_0))>0$. But for any $\Lambda\in \mathfrak{X}_{d+1}\setminus X_0$, $a_t \Lambda$ will eventually leave every compact set as $t\to +\infty$ by Mahler's criterion. As $\mu$ is $a_t$-invariant, this contradicts Poincar\'e's recurrence theorem. Therefore, $\mu(X_0)=0$.

To verify (ii), we prove a stronger assertion: for any $\Lambda\in\mathfrak{X}_{d+1}$, $\mathcal{Y}_{\Lambda}$ is discrete. First suppose to the contrary that there exists $\Lambda\in X_h\cup X_v$ such that $\mathcal{Y}_{\Lambda}$ is nondiscrete. Let $t_0$ be an accumulation point of $\mathcal{Y}_{\Lambda}$, then there exists a sequence $\{t_n\}_{n\in \mathbb{N}}$ such that $t_n \to t_0$ as $n\to \infty$. Without loss of generality, we may assume that $0<t_n-t_0<1$ for all $n\geq 1$. Consider the compact set
$ D_1^{[-1,0]}=\bigcup_{t\in [-1,0]} a_t D_1$ in $\mathbb{R}^{d+1}$. $D_1^{[-1,0]}\cap a_{t_0}\Lambda$ is a finite subset of $\mathbb{R}^{d+1}$. As for each $n\in \mathbb{N}$, $a_{t_n-t_0}\cdot a_{t_0}\Lambda \in \mathcal{S}_1$, there exists nonzero $\boldsymbol{v}_n\in a_{t_0}\Lambda$ such that $a_{t_n-t_0} \boldsymbol{v}_n\in D_1$. This implies $\boldsymbol{v}_n \in D_1^{[-1,0]}\cap a_{t_0}\Lambda$. Since $\boldsymbol{v}_i\neq \boldsymbol{v}_j$ if $i\neq j$, $ D_1^{[-1,0]}\cap a_{t_0}\Lambda$ is an infinite set, which is a contradiction. Therefore, $\mathcal{Y}_{\Lambda}$ is discrete. 

For $\Lambda\in X_0$, the fact that $\mathcal{Y}_{\Lambda}$ is unbounded from above and below follows from Lemma \ref{Lemma: the visit times are unbounded}, as $X_0=X_h\cap X_v$.

For (iii), we only need to show that for any $\epsilon>0$,
$\mathcal{S}_{\epsilon}=\{\Lambda \in \mathcal{S}_1: \tau(\Lambda)<\epsilon\}$
is measurable. Note that 
\begin{align*}
    \mathcal{S}_{\epsilon}&=\{\Lambda\in \mathcal{S}_1: \Lambda_{prim}\cap D_1^{(-\epsilon,0)}\}\\
    &=\mathfrak{X}_{d+1}(D_1)\cap \mathfrak{X}_{d+1}(D_1^{(-\epsilon,0)})
\end{align*}
is measurable.

For the last assertion, because $\mathfrak{X}_{d+1}(D_1,2)^{\mathbb{R}}\subset \mathfrak{X}_{d+1}\setminus X_0$, $\mu(\mathfrak{X}_{d+1}(D_1,2)^{\mathbb{R}})=0$. Therefore, $\mu_{\mathcal{S}_1}(\mathfrak{X}_{d+1}(D_1,2))=0$.
\end{proof}

From now on till the end of Section \ref{Section: cross sections}, we will let $\mu=\mu_{d+1}$ be the unique $G$-invariant probability measure on $\mathfrak{X}_{d+1}$.

\subsection{Parameterizing $\mathcal{S}_r$ for $r>0$}
We follow the notations in \cite{Shapira_Weiss_2022_Geometric_and_arithmetic_aspects}. Let
\[ H:=\left\{ \begin{pmatrix}
    A & 0\\
    ^t\boldsymbol{h} & 1
\end{pmatrix}: A\in SL_{d}(\mathbb{R}), \boldsymbol{h}\in \mathbb{R}^{d} \right\},\]
and 
\begin{equation}\label{equation: definition of U}
    U:=\left\{ u(\boldsymbol{v})=\begin{pmatrix}
    I_{d} & \boldsymbol{v}\\
    0 & 1
\end{pmatrix}: \boldsymbol{v}\in \mathbb{R}^{d} \right\}.
\end{equation}
For $1\leq i\leq d+1$, we denote by $\boldsymbol{e}_i\in \mathbb{R}^{d+1}$ the column vector with $1$ in the $i$-th row and $0$ otherwise.
Then it is clear that $H\mathbb{Z}^{d+1}$ can be identified with $\mathcal{E}_{d+1}=\mathfrak{X}_{d+1}(\boldsymbol{e}_{d+1})$ of unimodular lattices in $\mathbb{R}^{d+1}$ which contains $\boldsymbol{e}_{d+1}$ as a primitive vector. For $r>0$, let $\overline{B_r^{\boldsymbol{w}}}\subset \mathbb{R}^{d}$ be the closed ball centered at $\boldsymbol{0}\in \mathbb{R}^{d}$ with respect to the (quasi-)norm $\norm{\cdot}_{\boldsymbol{w}}$. Consider the map 
\begin{align*}
    \varphi: \mathcal{E}_{d+1}\times \overline{B^{\boldsymbol{w}}_{r}}\to \mathcal{S}_r\\
    (\Lambda,\boldsymbol{v})\mapsto \varphi(\Lambda,\boldsymbol{v})=u(\boldsymbol{v})\Lambda.
\end{align*}

The map $\boldsymbol{v}\mapsto u(\boldsymbol{v})\boldsymbol{e}_{d+1}$ is a bijection between $\overline{B^{\boldsymbol{w}}_{r}}$ and $D_r$. Thus $\varphi(\mathcal{E}_{d+1}\times \overline{B^{\boldsymbol{w}}_r})=\mathcal{S}_r$. Moreover for any $\Lambda\in \mathcal{S}_r$, 
\[\# \varphi^{-1}(\Lambda)=\#(\Lambda_{prim}\cap D_r).\]
This is because for any $\boldsymbol{v}\in \Lambda_{prim}\cap D_r$,
\[(u(-\pi_{\mathbb{R}^{d}}(\boldsymbol{v}))\Lambda, \pi_{\mathbb{R}^d}(\boldsymbol{v}))\in \varphi^{-1}(\Lambda).\]
Now let $r=1$. Define 
\begin{align*}
    \psi: \mathcal{S}_1^{\#}\to \mathcal{E}_{d+1}\times \overline{B^{\boldsymbol{w}}_1}\\
    \psi(\Lambda)=(u(-\boldsymbol{v}_{\Lambda})\Lambda,\boldsymbol{v}_{\Lambda}),
\end{align*}
where $\boldsymbol{v}_{\Lambda}=\pi_{\mathbb{R}^d}(\boldsymbol{v}(\Lambda))$, and $\{\boldsymbol{v}(\Lambda)\}=\Lambda\cap D_1$.

\begin{lemma}\cite[Lemma 8.5]{Shapira_Weiss_2022_Geometric_and_arithmetic_aspects}
    The set $\mathcal{S}_1$ is closed in $\mathfrak{X}_{d+1}$, $\psi$ is the inverse of $\varphi|_{\varphi^{-1}(\mathcal{S}_1^{\#}) }$, and a homeomophism between $\mathcal{S}_1^{\#}$ and $\varphi^{-1}(\mathcal{S}_1^{\#})$.
\end{lemma}

\begin{proof}
   We note that by \cite[Lemma 8.1]{Shapira_Weiss_2022_Geometric_and_arithmetic_aspects}, $\mathcal{S}_1$ is closed in $\mathfrak{X}_{d+1}$ and $\psi$ is continuous. It is clear that $\varphi\circ \psi=Id$, and $\psi\circ \varphi|_{Image(\psi)}=Id$. Therefore, the lemma follows.
\end{proof}

\subsection{Measure of cross-section}
From now on until the end of Section \ref{Section: cross sections}, we will focus on the cross-section $\mathcal{S}_1$.

\begin{lemma}\label{Lemma: cross section is Jordan measurable}
For any sufficiently small $\epsilon>0$, the set $\mathcal{S}_{1,<\epsilon}$ is $\mu_{\mathcal{S}_1}$-JM.
\end{lemma}
\begin{proof}
    The proof goes verbatim the same as that of \cite[Lemma 8.9]{Shapira_Weiss_2022_Geometric_and_arithmetic_aspects}.
\end{proof}
Now let $\mathcal{U}_1:=\mathfrak{X}_{d+1}^{\#}(D_1)\cap \mathfrak{X}_{d+1}(D^0_1)$.
\begin{lemma}\label{Lemma: a large open set in the section}
    The set $\mathcal{U}_1$ is open in $\mathcal{S}_1$, the set $(cl_{\mathfrak{X}_{d+1}}(\mathcal{S}_1)\setminus \mathcal{U}_1)^{(0,1)}$ is $\mu$-null, and the map $(t,\Lambda)\mapsto a_t\Lambda$ from $(0,1)\times \mathcal{U}_1$ is open.
\end{lemma}
\begin{proof}
    The proof of \cite[Lemma 8.10]{Shapira_Weiss_2022_Geometric_and_arithmetic_aspects} can be adapted verbatim to the proof of the above lemma.
\end{proof}

\begin{proposition}\label{Proposition: measure of cross section}
    $\mu_{\mathcal{S}_1}$ is finite and $supp(\mu_{\mathcal{S}_1})=\mathcal{S}_1$. Indeed, 
    \[\mu_{\mathcal{S}_1}=\frac{1}{\zeta(d+1)}\cdot \varphi_*(\mu_{\mathcal{E}_{d+1}}\times Leb_{\overline{B}_1^{\boldsymbol{w}}}),\]
    where $\mu_{\mathcal{E}_{d+1}}$ is the Haar measure on $\mathcal{E}_{d+1}$, and $\zeta$ is the Riemann zeta function. In particular, as for any other weight vector $\boldsymbol{w}'$, $B_1^{\boldsymbol{w}'}=B_1^{\boldsymbol{w}}$, $\mu_{\mathcal{S}_1}$ does not depends on $\boldsymbol{w}$.
\end{proposition}
\begin{proof}
    Let $Q=\{a_t:t\in \mathbb{R}\}\ltimes U$. Note that $G=Q\times H$ and $Q\cap H=\{e\}$. Then the proof goes verbatim as that of \cite[Proposition 8.7]{Shapira_Weiss_2022_Geometric_and_arithmetic_aspects}.
\end{proof}

\begin{theorem}\label{Theorem: cross section is reasonable}
    The cross section $\mathcal{S}_1$ is $\mu$-reasonable.
\end{theorem}
\begin{proof}
    Recalling the definition of reasonable cross section as in Definition \ref{Definition: resonable section}, the theorem follows by Lemma \ref{Lemma: cross section is Jordan measurable}, \ref{Lemma: a large open set in the section} and Proposition \ref{Proposition: measure of cross section}.
\end{proof}

\subsection{The set $\mathcal{B}$ is tempered}

\begin{lemma}\label{Lemma: boundary of best approximation set}\cite[Lemma 9.1]{Shapira_Weiss_2022_Geometric_and_arithmetic_aspects}
    The set $\mathcal{B}$ is open in $\mathcal{S}_1$. The boundary $\partial_{\mathcal{S}_1}\mathcal{B}$ is contained in $\mathfrak{X}_{d+1}(D_1,2)\cup \mathcal{Z}$, where 
    \[\mathcal{Z}=\{\Lambda\in \mathfrak{X}_{d+1}:\exists \boldsymbol{v},\boldsymbol{w}\in \Lambda_{prim} \text{ such that }\boldsymbol{v}\neq \pm \boldsymbol{w}, \norm{\pi_{\mathbb{R}^d}(\boldsymbol{v})}_{\boldsymbol{w}}=\norm{\pi_{\mathbb{R}^d}(\boldsymbol{w})}_{\boldsymbol{w}}\}.\]
\end{lemma}

\begin{proof}
The proof is verbatim the same as that of Lemma 9.1 in \cite{Shapira_Weiss_2022_Geometric_and_arithmetic_aspects}.
\end{proof}

\begin{lemma}\label{Lemma: B is JM}\cite[Lemma 9.2]{Shapira_Weiss_2022_Geometric_and_arithmetic_aspects}
    The set $\mathcal{B}$ is $\mu_{S_1}$-JM and $\mu_{S_1}(\mathcal{B})>0$.
\end{lemma}

\begin{proof}
The proof goes verbatim the same as that of Lemma 9.2 in \cite{Shapira_Weiss_2022_Geometric_and_arithmetic_aspects}.
\end{proof}

\begin{proposition}\label{Proposition: B is tempered}\cite[Proposition 9.8]{Shapira_Weiss_2022_Geometric_and_arithmetic_aspects}
    The set $\mathcal{B}$ is tempered.
\end{proposition}

\begin{proof}
    For any $r>0$, let
    \begin{align*}
    C_r(e)=\left\{\boldsymbol{v}={^t(}\boldsymbol{u},c)\in \mathbb{R}^{d+1}: \boldsymbol{u}\in \mathbb{R}^d, \norm{\boldsymbol{u}}_{\boldsymbol{w}}\leq r; c\in \mathbb{R},|c|\leq e \right\}.
    \end{align*}
By compactness of $C_r(e)$ and (\ref{equation: triangle inequality for quasi norm}), we can choose $M\in \mathbb{N}$ large enough (depending only on $w_d$, $d$) such that for any subset of $C_r(e)$ with cardinality at least $M+1$, there are distinct points $\boldsymbol{v}={^t(} \boldsymbol{u},c)$, and $\boldsymbol{v}'={^t(} \boldsymbol{u}',c')$ in $C_r(e)$ satisfying 
    \[\norm{\boldsymbol{u}'-\boldsymbol{u}}_{\boldsymbol{w}}\leq \frac{r}{3}, \text{ and } |c'-c|<1.\]
Now we claim that $\mathcal{B}$ is $M$-tempered, that is, every $\Lambda\in \mathcal{B}$ satisfies (\ref{equation: definition of temperedness}). Assume by contradiction that there exist $\Lambda\in \mathcal{B}$ and $0=t_0<t_1<\cdots<t_M\leq 1$ such that $a_{t_j}\Lambda \in \mathcal{B}$ for all $0\leq j \leq M$. Then for each $0\leq j \leq M$, the vector $\boldsymbol{w}_j=\boldsymbol{v}(a_{t_j}\Lambda)$ (see (\ref{equation: definition of v(lambda)})) satisfies 
\[\boldsymbol{w}_j\in a_{t_j}\Lambda_{prim}\cap D_1, \text{ and }a_{t_j}\Lambda \cap C^0_{\norm{\pi_{\mathbb{R}^d}(\boldsymbol{w}_j)}_{\boldsymbol{w}}}=\{\boldsymbol{0}\}.\]
Moreover, for each $0\leq j \leq M$, the vector
\[a_{-t_j}\boldsymbol{w}_j={^t(}\boldsymbol{v}_j,e^{t_j})\in \Lambda\]
satisfies $\Lambda\cap C^0_{\norm{\boldsymbol{v}_j}_{\boldsymbol{w}}}(e^{t_j})=\{\boldsymbol{0}\}$. As $t_0<t_1<\cdots<t_M$, this implies that 
\[\norm{\boldsymbol{v}_0}_{\boldsymbol{w}}\geq \norm{\boldsymbol{v}_1}_{\boldsymbol{w}}\geq \cdots\geq \norm{\boldsymbol{v}_M}_{\boldsymbol{w}}.\]
Therefore, we have 
\[{^t(} \boldsymbol{v}_j,
    e^{t_j})\in C_{\norm{\boldsymbol{v}_0}_{\boldsymbol{w}}}(e), \text{ for }j=0,\cdots, M.\]
By the choice of $M$, there exist $0\leq j_1<j_2\leq M$ such that
\[\norm{\boldsymbol{v}_{j_1}-\boldsymbol{v}_{j_2}}_{\boldsymbol{w}}<\frac{\norm{\boldsymbol{v}_0}_{\boldsymbol{w}}}{3}, \text{ and } 0<|e^{t_{j_1}}-e^{t_{j_2}}|<1.\]
Set $\boldsymbol{w}={^t(}
    \boldsymbol{v}_{j_1}-\boldsymbol{v}_{j_2},
    e^{t_{j_1}}-e^{t_{j_2}}).$
Then $\boldsymbol{w}\neq \boldsymbol{0}$ and $\boldsymbol{w}\in \Lambda\cap C^0_{\norm{\boldsymbol{v}_0}_{\boldsymbol{w}}}$. This contradicts the assumption that $\Lambda\in \mathcal{B}$.
\end{proof}

\subsection{Proof of Theorem \ref{Theorem: weighted Levy-Khintchine constant} (1) } 
Applying Theorem \ref{Theorem: Pointwise ergodic theorem with rate for nonsmooth functions} (or \cite[Theorem 1.1]{Kleinbock_Shi_Barak_2017_Pointwise_equidistribution_with_an_error_rate_and_with_respect_to_unbounded_functions_MR3606456}) to the identity coset in $\mathfrak{X}_{d+1}$ and $U$ ($U$ is as in (\ref{equation: definition of U})), we obtain that for Lebesgue almost every $\boldsymbol{\theta}\in \mathbb{R}^d$, $\Lambda_{\boldsymbol{\theta}}=u(-\boldsymbol{\theta})\mathbb{Z}^{d+1}$ is generic with respect to $(a_t,\mu)$.

As $\mathcal{S}_1$ is a $\mu$-reasonable cross-section by Theorem \ref{Theorem: cross section is reasonable}, and $\mathcal{B}\subset \mathcal{S}_1$ is a $\mu_{\mathcal{S}_1}$-JM tempered set by Lemma \ref{Lemma: B is JM} and Proposition \ref{Proposition: B is tempered}, Proposition \ref{Proposition: limit behavior of visiting times to tempered set} applies and we obtain that for Lebesgue almost every $\boldsymbol{\theta}\in \mathbb{R}^d$,
\begin{equation}\label{equation: limit of visit times of Lambda theta}
    \lim_{T\to \infty}\frac{1}{T} N(\Lambda_{\boldsymbol{\theta}},T,\mathcal{B})=\mu_{\mathcal{S}_1}(\mathcal{B}).
\end{equation}
In particular, since $\mu_{\mathcal{S}_1}(\mathcal{B})>0$ by Lemma \ref{Lemma: B is JM}, for any $\boldsymbol{\theta}$ satisfying (\ref{equation: limit of visit times of Lambda theta}), we conclude 
\[\{t>0:a_t \Lambda_{\boldsymbol{\theta}}\in \mathcal{B}\}=\{t_1<t_2<\cdots\}\]
is an infinite set.  Then Corollary \ref{Corollary: two sequences are prefix equivalent} implies that $\{q'_n(\boldsymbol{\theta})\}_{n\in \mathbb{N}}$ and $\{q_n(\boldsymbol{\theta})\}_{n\in\mathbb{N}}$ are prefix-equivalent, where $q_n'(\boldsymbol{\theta})=e^{t_n}$, and $q_n(\boldsymbol{\theta})$ is the last row of the $n$-th $\boldsymbol{w}$-best approximable vector of $\boldsymbol{\theta}$.

Note that for any $\boldsymbol{\theta}$ satisfying (\ref{equation: limit of visit times of Lambda theta}), we have
\begin{equation*}
    \# \{t\in (0,t_n]: a_t \Lambda_{\boldsymbol{\theta}}\in \mathcal{B}\}=n, \quad \forall n\in \mathbb{N}.
\end{equation*}
As $\mathcal{B}$ is tempered, $t_n\to\infty$ as $n\to \infty$, and thus (\ref{equation: limit of visit times of Lambda theta}) implies that for Lebesgue almost every $\boldsymbol{\theta}\in \mathbb{R}^d$,
\begin{equation*}
    \frac{n}{\log q_n'(\boldsymbol{\theta})}=\frac{n}{t_n}=\frac{1}{t_n} \# \{t\in (0,t_n]: a_t \Lambda_{\boldsymbol{\theta}}\in \mathcal{B}\}\to \mu_{\mathcal{S}_1}(\mathcal{B}), \text{ as }n\to \infty.
\end{equation*}
It is clear that as $\{q'_n(\boldsymbol{\theta})\}_{n\in \mathbb{N}}$ and $\{q_n(\boldsymbol{\theta})\}_{n\in\mathbb{N}}$ are prefix-equivalent, we have
\begin{equation*}
    \lim_{n\to \infty}\frac{\log q_n(\boldsymbol{\theta})}{n}=\lim_{n\to \infty}\frac{\log q_n'(\boldsymbol{\theta})}{n}=\frac{1}{\mu_{\mathcal{S}_1}(\mathcal{B})}.
\end{equation*}
Set $L_d(\boldsymbol{w})=1/\mu_{\mathcal{S}_1}(\mathcal{B})$, then (\ref{equation: limit of qn}) follows.

\section{The consequence of effective equidistribution and Borel-Cantelli Lemma}\label{Section: Borel-Cantelli}

In this section, we follow the strategy of \cite{Kleinbock_Margulis_1999_Logarithm_laws_MR1719827} and \cite{Cheung_Chevallier_2019_Levy_Khintchin_Theorem} to prove (2) (3) of Theorem \ref{Theorem: weighted Levy-Khintchine constant}.

Let $\Lambda\in \mathcal{B}$. Let $r(\Lambda)$ be as in (\ref{equation: definition of r(lambda)}), and $t(\Lambda)>0$ be the smallest $t>0$ such that $a_t\Lambda\in \mathcal{B}$. Now we investigate the quantity $e^{t(\Lambda)}r(\Lambda)$. Consider the symmetric convex body
\[\mathcal{C}(\Lambda):=\{\boldsymbol{v}\in \mathbb{R}^{d+1}: \norm{\pi_{\mathbb{R}^d}(\boldsymbol{v})}_{\boldsymbol{w}}<r(\Lambda), |\boldsymbol{e}_{d+1}\cdot \boldsymbol{v}|<e^{t(\Lambda)}\},\]
where $\boldsymbol{e}_{d+1}\cdot \boldsymbol{v}$ is the last row of $\boldsymbol{v}$. By the choice of $t(\Lambda)$, $\Lambda\cap \mathcal{C}(\Lambda)=\{\boldsymbol{0}\}$. Then as the volume of $\mathcal{C}(\Lambda)$ is $2^{d+1}r(\Lambda)e^{t(\Lambda)}$, Minkowski's first theorem implies 
\begin{equation}\label{equation: upperbound for e t(lambda) r(lambda)}
    e^{t(\Lambda)}r(\Lambda)\leq 1.
\end{equation}

In particular, for any $\boldsymbol{\theta}\in \mathbb{R}^d$, as before we let $^t(\boldsymbol{p}_n(\boldsymbol{\theta}),q_n(\boldsymbol{\theta}))={^t(}\boldsymbol{p}_n,q_n)\in \mathbb{Z}^d\times \mathbb{N}$ be the $n$-th $\boldsymbol{w}$-best approximable vector of $\boldsymbol{\theta}$. Let $r_n=r_n(\boldsymbol{\theta})=\norm{q_n\boldsymbol{\theta}-\boldsymbol{p}_n}_{\boldsymbol{w}}$. Note that $q_n \cdot r_n=\norm{\pi_{\mathbb{R}^d}(a_{t_n} u(-{\boldsymbol{\theta}}){^t(}\boldsymbol{p}_n,q_n))}_{\boldsymbol{w}}$, where $t_n=\log q_n$, and that 
\begin{equation}\label{equation: relation between two notions}
    r(a_{t_n}\Lambda_{\boldsymbol{\theta}})=q_n r_n, \text{ and } t(a_{t_n}\Lambda_{\boldsymbol{\theta}})=\log(q_{n+1}/q_n).
\end{equation}
Thus, by (\ref{equation: upperbound for e t(lambda) r(lambda)}), we have
\begin{equation}\label{equation: upper bound for qn times rn}
   q_n(\boldsymbol{\theta})r_n(\boldsymbol{\theta}) \leq q_{n+1}(\boldsymbol{\theta})r_n(\boldsymbol{\theta})\leq 1.
\end{equation}
We are ready to prove Theorem \ref{Theorem: weighted Levy-Khintchine constant} (3):
\begin{proof}[Proof of Theorem \ref{Theorem: weighted Levy-Khintchine constant} (3)]
    Define a map $F:\mathcal{B}\to \mathbb{R}$ by
\[F(\Lambda)=e^{t(\lambda)} r(\Lambda) \]
Let $f:\mathbb{R}\to \mathbb{R}$ be a continuous and bounded function. Then by (\ref{equation: upperbound for e t(lambda) r(lambda)}), $f\circ F:\mathcal{B}\to \mathbb{R}$ is a continuous and bounded function. Denote by $R:\mathcal{B}\to \mathcal{B}$ by the first return map induced by $a_t$. Let $\boldsymbol{\theta}\in \mathbb{R}^d$ be such that $\{a_t\Lambda_{\boldsymbol{\theta}}:t>0\}\cap \mathcal{B}$ is infinite. Then by (\ref{equation: relation between two notions}) and Corollary \ref{Corollary: two sequences are prefix equivalent}, there exists $k_0\in \mathbb{N}$ such that for all $i\in \mathbb{N}$, 
\[F(R^i \Lambda_{\boldsymbol{\theta}})=q_{k_0+i}(\boldsymbol{\theta})r_{k_0+i}(\boldsymbol{\theta})=\beta_{k_0+i}(\boldsymbol{\theta}).\]
By Proposition \ref{Proposition: limit behavior of visiting times to tempered set}, for almost all $\boldsymbol{\theta}\in \mathbb{R}^d$, $\Lambda_{\boldsymbol{\theta}}$ is $(a_t,\mu_{\mathcal{S}_1}|_{\mathcal{B}})$-generic. Therefore, for almost all $\boldsymbol{\theta}$,
\begin{align*}
    \lim_{n\to\infty}\frac{1}{n}\sum_{i=0}^{n-1} f\circ F(R^i \Lambda_{\boldsymbol{\theta}})=\frac{1}{\mu_{\mathcal{S}_1}(\mathcal{B})}\int_{\mathcal{B}} f\circ F(\Lambda) d\mu_{\mathcal{S}_1}(\Lambda
    )=\int_{\mathbb{R}} f(x)d\nu_d^{\boldsymbol{w}}.
\end{align*}
This finishes the proof.
\end{proof}

Note that (\ref{equation: upper bound for qn times rn}) gives an upper bound for $q_n(\boldsymbol{\theta})r_n(\boldsymbol{\theta})$ for any $\boldsymbol{\theta}\in \mathbb{R}^d$. In order to prove Theorem \ref{Theorem: weighted Levy-Khintchine constant} (2), we have to obtain an lower bound of $q_n(\boldsymbol{\theta})r_n(\boldsymbol{\theta})$ for suitable $\boldsymbol{\theta}$. To do so, we consider a decreasing function $\varphi:\mathbb{N}\to \mathbb{R}_+$ defined by
 \begin{align*}
     \varphi(n)=n^{-\alpha}
 \end{align*}
where $\alpha>\frac{2}{(1+d)w_d}$ is a constant. Moreover, we define a function $s:\mathbb{N}\to \mathbb{R}$ by
\[s(n)=n-\frac{\log \varphi(n)}{2}.\]
For any $\epsilon>0$, define 
\[K_{\epsilon}:=\{\Lambda\in \mathfrak{X}_{d+1}: \forall \boldsymbol{v}\in \Lambda\setminus\{0\}, \norm{\boldsymbol{v}}_{\infty}\geq \epsilon\},\]
where $\norm{\cdot}_{\infty}$ is the usual supremum norm on $\mathbb{R}^{d+1}$.
Then by Mahler's compactness criterion, $K_{\epsilon}$ is compact. Note that Siegal's integral formula implies
\begin{equation}\label{equation: measure of cusp nhbd}
    1-\mu_{d+1}(K_{\epsilon})\ll_d \epsilon^{d+1}.
\end{equation}

\begin{lemma}\label{Lemma: consequence of qn rn small}
    Given $\boldsymbol{\theta}\in \mathbb{R}^d$ and $n\in \mathbb{N}$, let $k_n=k_n(\boldsymbol{\theta})=\lfloor \log q_n(\boldsymbol{\theta}) \rfloor$. Assume there exists $c>1$ such that $k_n\leq cn$. If there exists $n\in \mathbb{N}$ such that $q_n(\boldsymbol{\theta})r_n(\boldsymbol{\theta})\leq \varphi(n)$, then 
    \begin{equation*}
a_{s(k_n)}\Lambda_{\boldsymbol{\theta}}\in \mathfrak{X}_{d+1}\setminus K_{c(k_n)},
    \end{equation*}
    where $c(k_n)=e c^{\alpha w_1}\varphi(k_n)^{w_d/2}$.
\end{lemma}

\begin{proof}
As $\boldsymbol{\theta}$ is fixed, in the following we will write $q_n(\boldsymbol{\theta})=q_n$ and $r_n(\boldsymbol{\theta})=r_n$ for simplicity. Since 
    \[r_n=\norm{q_n\boldsymbol{\theta}-\boldsymbol{p}_n}_{\boldsymbol{w}}=\norm{\pi_{\mathbb{R}^d}(u(-\boldsymbol{\theta}){^t(}\boldsymbol{p}_n,q_n))}_{\boldsymbol{w}},\]
we have 
\[\norm{\pi_{\mathbb{R}^d}(a_{s(k_n)}u(-{\boldsymbol{\theta}}){^t(}\boldsymbol{p}_n,q_n))}_{\boldsymbol{w}}=e^{s(k_n)} r_n\leq \varphi(n)\varphi(k_n)^{-1/2}\leq c^{\alpha}\varphi(k_n)^{1/2}.\]
In particular, this implies that 
\[\norm{\pi_{\mathbb{R}^d}(a_{s_n}u(-{\boldsymbol{\theta}}){^t(}\boldsymbol{p}_n,q_n))}_{\infty}\leq c^{\alpha w_1}\varphi(k_n)^{w_d/2}.\]
On the other hand, we also have
\[e^{-s_n}q_n\leq e\varphi(k_n)^{1/2}\leq e \varphi(k_n)^{w_d/2}.\]
We conclude that 
\[\norm{a_{s(k_n)}u(-{\boldsymbol{\theta}}){^t(}\boldsymbol{p}_n,q_n)}_{\infty}\leq e c^{\alpha w_1}\varphi(k_n)^{w_d/2},\]
and the lemma follows.
\end{proof}
Now we want to show that the set $\mathfrak{C}$ defined by
\begin{equation*}
    \mathfrak{C}:=\{\boldsymbol{\theta}\in \mathbb{R}^d: a_{s(n)}\Lambda_{\boldsymbol{\theta}}\in \mathfrak{X}_{d+1}\setminus K_{c(n)}\text{ for infinitely many } n\in \mathbb{N}\}
\end{equation*}
satisfies $Leb(\mathfrak{C})=0$,
where 
\[c(n)=e c^{\alpha w_1}\varphi(n)^{w_d/2},\]
and $c$ is a fixed constant greater than $1$.
For each $n\in \mathbb{N}$, define
\begin{equation*}
    \mathfrak{C}_n:=\{\boldsymbol{\theta}\in [0,1]^d: a_{s(n)}\Lambda_{\boldsymbol{\theta}}\in \mathfrak{X}_{d+1}\setminus K_{c(n)}\}.
\end{equation*}
Then it is clear that $Leb(\mathfrak{C})=0$ would follow from the following
\begin{proposition}\label{Proposition: limsup set has measure 0}
    $Leb(\limsup_{n\to \infty}\mathfrak{C}_n)=0$.
\end{proposition}

\begin{proof}
We follow the strategy of Kleinbock-Margulis \cite{Kleinbock_Margulis_1999_Logarithm_laws_MR1719827} (see also \cite{Kleinbock_Strombergsson_Yu_2022_A_measure_estimate_MR4500198}). For each $n\in \mathbb{N}$, we denote by $R_n$ the injectivity radius of $K_{c(n)}$. Then (see e.g. \cite[Proposition 3.5]{Kleinbock_Margulis_2012_On_effective_MR2867926})
\begin{equation}\label{equation: lower bound for injectivity radius}
    R_n\gg_d c(n)^{d+1}.
\end{equation}
Also since $c(n)\to 0$ as $n\to \infty$, $R_n\to 0$.

For any $r>0$, let $\theta_r\in C_c^{\infty}(G)$ be the bump function given by Lemma \ref{Lemma: bump function}. For each $n\in \mathbb{N}$, define
\[f_n=\theta_{R_n}*\chi_{K_{2c(n)}}\in C_c^{\infty}(\mathfrak{X}_{d+1}).\]
Note that for any $g\in \mathcal{O}_{1/2}$ (see (\ref{equation: definition of neighborhood of G})), and $\boldsymbol{v}\in \mathbb{R}^{d+1}$, we have 
\[\frac{1}{2}\norm{\boldsymbol{v}}\leq \norm{g\boldsymbol{v}}\leq 2 \norm{\boldsymbol{v}}.\]
As $R_n\to 0$ as $n\to \infty$, for all $n$ large enough, $R_n<1/2$. Therefore, for all $n$ large enough, we have
\begin{equation}\label{equation: inequality about fn}
    \chi_{K_{4c(n)}}\leq f_n\leq \chi_{K_{c(n)}}.
\end{equation}
By Corollary \ref{Corollary: effective equidistribution}, there exists $l\in \mathbb{N}$ and $\delta>0$ such that for all $n\in \mathbb{N}$,
\begin{equation*}
    \int_{[0,1]^d} f_n(a_{s(n)}\Lambda_{\boldsymbol{\theta}}) d\boldsymbol{\theta}=\mu_{d+1}(f_n)+O(e^{-\delta s(n)}\mathcal{N}_l(f_n)),
\end{equation*}
where $\mu_{d+1}$ is the $G$-invariant probability measure on $\mathfrak{X}_{d+1}$. Therefore, 
\begin{align*}
\int_{[0,1]^d}\chi_{K_{c(n)}}(a_{s(n)}\Lambda_{\boldsymbol{\theta}})d\boldsymbol{\theta}
&\geq \int_{[0,1]^d} f_n(a_{s(n)}\Lambda_{\boldsymbol{\theta}}) d\boldsymbol{\theta}\\
&=\mu_{d+1}(f_n)+O(e^{-\delta s(n)}\mathcal{N}_l(f_n))\\
&\geq \mu_{d+1}(K_{4c(n)})+O(e^{-\delta s(n)}\mathcal{N}_l(f_n)),
\end{align*}
and we obtain
\begin{align}\label{align: inequality for Cn}
    Leb(\mathfrak{C}_n)\leq 1- \mu_{d+1}(K_{4c(n)})+O(e^{-\delta s(n)}\mathcal{N}_l(f_n)).
\end{align}
We now estimate $\mathcal{N}_l(f_n)$ for large enough $n$. First of all, by definition of $f_n$, $\norm{f_n}_{\infty}\leq 1$.

Denote by $\mu_G$ the Haar measure on $G$. Using Lemma \ref{Lemma: bump function}, the fact that $supp(\theta_{R_n})\subset \mathcal{O}_{R_n}$, and $\mu_G(\mathcal{O}_{R_n})\asymp_d R_n^{(d+1)^2-1}$, for any differential operator $Z$ on $C_c^{\infty}(\mathfrak{X}_{d+1})$ of degree $\deg(Z)\leq l$ we have
\begin{align*}
    \int |\mathcal{D}_Z(f_n)|^2 d\mu_{d+1}
    &=\int |\mathcal{D}_Z(\theta_{R_n})*\chi_{K_{2c(n)}}|^2 d\mu_{d+1}\\
    &\leq \left(\int|\mathcal{D}_Z(\theta_{R_n})(g)d\mu_G(g)| \right)^2\\
    &\leq \left(\mu_G(\mathcal{O}_{R_n})\cdot\norm{\mathcal{D}_Z(\theta_{R_n})}_{L^{\infty}} \right)^2\\
    &\ll_{d,l}R_n^{-2l}.
\end{align*}
By definition of Sobolev norm (see (\ref{equation: definition of Sobolev norm})), we obtain
\begin{equation*}
    \norm{f_n}_{H^l}\ll_{l,d} R_n^{-2l}.
\end{equation*}
It remains to estimate the Lipschitz norm of $f_n$. Let $h\in G$ be such that $\dist_G(h,Id)<R_n$. When $R_n$ is sufficiently small, $\dist_G(h,Id)\asymp_d \norm{h-Id}$. In particular, there exists a positive constant $C$ (depending only on $d$) such that
\begin{equation*}
    C\cdot \dist_G(h,Id)\geq \norm{h-Id}.
\end{equation*}
As a result, $h\in \mathcal{O}_{CR_n}$, and by Lemma \ref{Lemma: inclusion about Or},
\begin{equation}\label{equation: h O subset of}
    h\mathcal{O}_{R_n}\subset \mathcal{O}_{2(C+1)R_n}.
\end{equation}
Moreover, by Lemma \ref{Lemma: bump function}, when $R_n$ is sufficiently small, $supp(\theta_{R_n})$ is contained in a fixed compact set (say $\mathcal{O}_{1/2}$), we have
\begin{equation}\label{equation: Lipschitz norm of theta}
    \sup_{g_1\neq g_2}\left\{ \frac{|\theta_{R_n}(g_1)-\theta_{R_n}(g_2)|}{\dist_G(g_1,g_2)}\right\}\ll_d \sup_{\deg(Z)=1}\norm{\mathcal{D}_Z(\theta_{R_n})}_{L^{\infty}}
    \ll_d R_n^{-(1+d)^2}.
\end{equation}
Now let $n\in \mathbb{N}$ be sufficiently large such that $R_n$ is small enough to satisfy (\ref{equation: h O subset of}) and (\ref{equation: Lipschitz norm of theta}). Now for any $x,y\in \mathfrak{X}_{d+1}$, if $\dist(x,y)\geq R_n$, we obtain
\begin{equation*}
    \frac{|f_n(x)-f_n(y)|}{\dist(x,y)}\leq 2 R_n^{-1}.
\end{equation*}
If $\dist(x,y)<R_n$, let $h\in G$ be such that $y=hx$ and $\dist(x,y)=\dist_G(h,Id)$. 
Therefore, by (\ref{equation: h O subset of}), (\ref{equation: Lipschitz norm of theta}) and $G$-invariance of $\mu_G$, we have
\begin{align*}
 |f_n(x)-f_n(y)|&=\left|\int \theta_{R_n}(g)\chi_{2c(n)}(g^{-1}x)d\mu_G-\int \theta_{R_n}(g)\chi_{2c(n)}(g^{-1}hx)d\mu_G \right|\\
 &=\left|\int (\theta_{R_n}(g)- \theta_{R_n}(hg))\chi_{2c(n)}(g^{-1}x)d\mu_G \right|\\
 &\leq \int_{\mathcal{O}_{2(C+1)R_n}} \frac{\left|\theta_{R_n}(g)-\theta_{R_n}(hg) \right|}{\dist_G(h,Id)}\cdot \dist_G(h,Id) d\mu_G\\
 &\ll_d \mu_G(\mathcal{O}_{2(C+1)R_n})\cdot R_n^{-(1+d)^2}\cdot\dist_G(h,Id)\\
 &\ll_d R_n^{-1}\cdot \dist_G(h,Id).
\end{align*}
Therefore, for all sufficiently large $n$, we have 
\[\norm{f_n}_{\Lip}\ll_d R_n^{-1}.\]
To conclude, there is a constant $M>0$ such that for all sufficiently large $n\in \mathbb{N}$,
\begin{equation}\label{equation: estimate of Nl}
    \mathcal{N}_l(f_n)\ll_d R_n^{-M}.
\end{equation}
Now by (\ref{equation: measure of cusp nhbd}), (\ref{equation: lower bound for injectivity radius}), (\ref{align: inequality for Cn}) and (\ref{equation: estimate of Nl}), there exists $n_0\in \mathbb{N}$ such that for all $n\geq n_0$,
\begin{align*}
    Leb(\mathfrak{C}_n)&\ll_d c(n)^{d+1}+e^{-\delta s(n)}\cdot R_n^{-M}\\
    &\ll_{l,d} n^{-(d+1)w_d\alpha/2}+e^{-\delta n}p(n),
\end{align*}
where
\[p(n)=n^{\frac{(d+1)w_dM\alpha-\delta \alpha}{2}}.\]
Recall that $\alpha$ is chosen to satisfy $\alpha>\frac{2}{(d+1)w_d}$. Since $p(n)$ grows polynomially as $n$ grows, $e^{-\delta n}p(n)\leq e^{-\delta n/2}$ for all $n$ large enough. Therefore, we have
\begin{align*}
    \sum_{n=1}^{\infty} Leb(\mathfrak{C}_n)\ll_{l,d} \sum_{n=1}^{\infty} ( n^{-(d+1)w_d\alpha/2}+e^{-\delta n}p(n))<\infty.
\end{align*}
By Borel-Cantelli lemma, the proposition follows.
\end{proof}

\begin{proof}[Proof of Theorem \ref{Theorem: weighted Levy-Khintchine constant} (2)]

Define the set
\begin{align*}
    \mathfrak{F}:=\{\boldsymbol{\theta}\in \mathbb{R}^d:\lim_{n\to \infty}\frac{1}{n}\log q_n(\boldsymbol{\theta})=1/{\mu_{\mathcal{S}_1}(\mathcal{B})}\}.
\end{align*}
Then $\mathfrak{F}$ has full Lebesgue. Moreover, let $c=\min\{2,2/\mu_{\mathcal{S}_1}(\mathcal{B})\}$. For any $\boldsymbol{\theta}\in \mathfrak{F}$, there exists $n(\boldsymbol{\theta})\in \mathbb{N}$ such that for all $n\geq n(\boldsymbol{\theta})$, $k_n(\boldsymbol{\theta})=\lfloor \log q_n(\boldsymbol{\theta})\rfloor\leq cn.$ Then Lemma \ref{Lemma: consequence of qn rn small} implies that 
\[\{\boldsymbol{\theta}\in \mathfrak{F}: q_n(\boldsymbol{\theta})r_n(\boldsymbol{\theta})\leq \varphi(n) \text{ for infinitely many }n\}\subset \mathfrak{F}\cap\mathfrak{C},\]
which is a Lebesgue null set by Proposition \ref{Proposition: limsup set has measure 0}. We conclude that for Lebesgue almost every $\boldsymbol{\theta}$, there are only finitely many $n\in \mathbb{N}$ such that 
\[q_n(\boldsymbol{\theta})r_n(\boldsymbol{\theta})\leq \varphi(n).\]
Therefore, for Lebesgue almost every $\boldsymbol{\theta}$, we have 
\[\liminf_{n\to \infty}\frac{1}{n}\log r_n(\boldsymbol{\theta})\geq -\lim_{n\to\infty}\frac{1}{n}\log q_n(\boldsymbol{\theta}).\]
On the other hand, by (\ref{equation: upper bound for qn times rn}) we have
\[\limsup_{n\to \infty}\frac{1}{n}\log r_n(\boldsymbol{\theta})\leq -\lim_{n\to\infty}\frac{1}{n}\log q_n(\boldsymbol{\theta}).\]
Thus, for Lebesgue almost every $\boldsymbol{\theta}$,
\[\lim_{n\to \infty}\frac{1}{n}\log r_n(\boldsymbol{\theta})= -\lim_{n\to\infty}\frac{1}{n}\log q_n(\boldsymbol{\theta}).\]
\end{proof}

\section{Appendix}
Let $\boldsymbol{w}={^t(}w_1,\cdots,w_d)\in \mathbb{R}^d$ be a weight vector and as before, we define the one-parameter diagonal flow $\{a_t\}$ as in (\ref{equation: definition of at from w}). We may assume that $w_1\geq \cdots \geq w_d>0$. Let $\norm{\cdot}_{\infty}$ be the usual supremum norm on $\mathbb{R}^{d+1}$, and $\norm{\cdot}_{\boldsymbol{w},\infty}$ be a quasi-norm defined by 
\[\norm{\boldsymbol{v}}_{\boldsymbol{w},\infty}=\max\{\norm{\pi_{\mathbb{R}^d}(\boldsymbol{v})}_{\boldsymbol{w}},|\boldsymbol{e}_{d+1}\cdot \boldsymbol{v}|\},\]
where $\norm{\cdot}_{\boldsymbol{w}}$ is the $\boldsymbol{w}$-quasi-norm as in (\ref{equation: definition of w-quasi norm}), and $\boldsymbol{e}_{d+1}\cdot \boldsymbol{v}$ is the last row of $\boldsymbol{v}$. For $\Lambda\in \mathfrak{X}_{d+1}$, we denote
\begin{align*}
\lambda_1(\Lambda)=\inf_{\boldsymbol{v}\in \Lambda\setminus\{\boldsymbol{0}\}}\norm{\boldsymbol{v}}_{\infty}, \quad \lambda_1^{\boldsymbol{w}}(\Lambda)=\inf_{\boldsymbol{v}\in \Lambda\setminus\{\boldsymbol{0}\}}\norm{\boldsymbol{v}}_{\boldsymbol{w},\infty}.
\end{align*}

\begin{lemma}\label{Lemma: difference between two minimas}
  For any $\Lambda\in \mathfrak{X}_N$, we have 
    \begin{equation}
      \log \lambda^{\boldsymbol{w}}_1(\Lambda) \asymp_{\boldsymbol{w}} \log \lambda_1(\Lambda).
    \end{equation}
\end{lemma}
\begin{proof}
    Let $\boldsymbol{v}^r,\boldsymbol{v}^w\in\Lambda\setminus\{\boldsymbol{0}\}$ denote the vectors realizing $\lambda_1(\Lambda)$ and $\lambda_1^{\boldsymbol{w}}(\Lambda)$, respectively. By Minkowski's first theorem, we have $\lambda_1(\Lambda),\lambda_1^{\boldsymbol{w}}(\Lambda)\leq 1$. Therefore,  
    \begin{equation*}
       \lambda_1(\Lambda)= \norm{\boldsymbol{v}^r}_{\infty}\leq \norm{\boldsymbol{v}^w}_{\infty} \leq \norm{\boldsymbol{v}^w}_{\boldsymbol{w},\infty}= \lambda_1^{\boldsymbol{w}}(\Lambda)^{w_d}
    \end{equation*}
    and 
    \begin{equation*}
        \lambda_1^{\boldsymbol{w}}(\Lambda)=\norm{\boldsymbol{v}^w}_{\boldsymbol{w},\infty}\leq \norm{\boldsymbol{v}^r}_{\boldsymbol{w},\infty}\leq \norm{\boldsymbol{v}^r}_{\infty}^{1/w_1}= \lambda_1(\Lambda)^{1/w_1}.
    \end{equation*}
The lemma follows by the above estimates.
\end{proof}

In Question 3 of \cite{Cheung_Chevallier_2019_Levy_Khintchin_Theorem}, Cheung and Chevallier considered the regular\footnote{In \cite{Cheung_Chevallier_2019_Levy_Khintchin_Theorem}, the authors didn't call this regular best approximation. Here we call it regular to distinguish the notion of best approximation of Cheung-Chevallier and $\boldsymbol{w}$-best approximation.} best approximable vectors of $\boldsymbol{\theta}\in \mathbb{R}^d$ with respect to $\{a_t\}$ defined as follows. 
A nonzero vector $(\boldsymbol{p},q)\in \mathbb{Z}^d\times \mathbb{N}$ is a regular best approximable vector of $\boldsymbol{\theta}$ if there exists $t>0$ such that 
\[\lambda_1(a_t \Lambda_{\boldsymbol{\theta}})=\norm{a_t u(-\boldsymbol{\theta}){^t(}\boldsymbol{p},q)}_{\infty}.\]
To compare the two definitions of best approximations, we begin with the following direct consequence of Definition \ref{Definition: w best approximation vector} and Proposition \ref{Proposition: dynamical formulation of best approximable vectors}.
\begin{lemma}
    Let $\boldsymbol{\theta}\in \mathbb{R}^d$ and $^t(\boldsymbol{p},q)\in \mathbb{Z}^d\times \mathbb{N}$ with $q>2^{1/w_d}$. Then $^t(\boldsymbol{p},q)$ is a $\boldsymbol{w}$-best approximable vector of $\boldsymbol{\theta}$ if and only if there exists $t>0$ such that 
    \[\lambda_1^{\boldsymbol{w}}(a_t \Lambda_{\boldsymbol{\theta}})=\norm{a_t u(-\boldsymbol{\theta}){^t(}\boldsymbol{p},q)}_{\boldsymbol{w},\infty}.\]
\end{lemma}
\begin{proof}
Assume $^t(\boldsymbol{p},q)$ is a $\boldsymbol{w}$-best approximable vector of $\boldsymbol{\theta}$. Let $t>0$ be such that
\[\norm{\pi_{\mathbb{R}^d}(a_t u(-\boldsymbol{\theta}){^t(}\boldsymbol{p
},q))}_{\boldsymbol{w}}=|e^{-t}q|.\]
Then by definition of $w$-best approximable vectors, it is directly verified that $\lambda_1^{\boldsymbol{w}}(a_t \Lambda_{\boldsymbol{\theta}})=\norm{a_t u(-\boldsymbol{\theta}){^t(}\boldsymbol{p},q)}_{\boldsymbol{w},\infty}$.

The proof of other direction is also straightforward, and we leave it to the reader.
\end{proof}
By discreteness of $\Lambda_{\boldsymbol{\theta}}$ in $\mathbb{R}^{d+1}$, it is clear that if $^t(\boldsymbol{p}^w,q^w)$ is a $\boldsymbol{w}$-best approximable vector of $\boldsymbol{\theta}$ ( $^t(\boldsymbol{p}^r,q^r)$ is a regular best approximable vector of $\boldsymbol{\theta}$, resp.), then there exists $t,\epsilon>0$ such that $\lambda_1^{\boldsymbol{w}}(a_{t'} \Lambda_{\boldsymbol{\theta}})=\norm{a_{t'} u(-\boldsymbol{\theta}){^t(}\boldsymbol{p}^w,q^w)}_{\boldsymbol{w},\infty}$ ( $\lambda_1(a_{t'} \Lambda_{\boldsymbol{\theta}})=\norm{a_{t'} u(-\boldsymbol{\theta}){^t(}\boldsymbol{p}^r,q^r)}_{\infty}$, resp.) for all $t'\in (t-\epsilon,t+\epsilon)$. Although Proposition \ref{Proposition: dynamical formulation of best approximable vectors} allows us to identify $\boldsymbol{w}$-best approximable vector $(\boldsymbol{p}^w,q^w)$ through the corresponding cross-section, by Lemma \ref{Lemma: difference between two minimas}, $(\boldsymbol{p}^w,q^w)$ is not necessary a regular best approximable vector, and vice-versa. Unlike Proposition \ref{Proposition: dynamical formulation of best approximable vectors}, it seems that regular best approximation does not admit a nice description through visiting times to certain cross section. %Question 3 in \cite{Cheung_Chevallier_2019_Levy_Khintchin_Theorem} seems to be out of reach of current methods.

\bibliographystyle{plain}
\bibliography{references.bib}

\end{document}